\documentclass{aims} 
\usepackage{amsmath}
\usepackage{amssymb}
\usepackage{amsfonts}
\usepackage{amsthm}
\usepackage{latexsym}
\usepackage{stmaryrd}
\usepackage{dsfont}
\usepackage{paralist}
\usepackage[misc]{ifsym}
\usepackage{epsfig} 
\usepackage{epstopdf} 
\usepackage[colorlinks=true]{hyperref}
\hypersetup{urlcolor=blue, citecolor=red}
\allowdisplaybreaks

\def\currentvolume{X}
 \def\currentissue{X}
  \def\currentyear{200X}
   \def\currentmonth{XX}
    \def\ppages{X-XX}
     \def\DOI{10.3934/xx.xxxxxxx}

\newtheorem{theorem}{Theorem}[section]

\newtheorem{lemma}[theorem]{Lemma}

\theoremstyle{definition}

\newtheorem{remark}[theorem]{Remark}

\newcommand{\tr}{{\rm tr}}
\newcommand{\id}{{\rm id}}
\renewcommand{\div}{{\rm div}}
\newcommand{\curl}{{\rm curl}}
\newcommand{\as}{{\rm as}}
\renewcommand{\S}{\mathds{S}}
\newcommand{\R}{\mathds{R}}

\newcommand{\Z}{\mathds{Z}}

\newcommand{\cS}{{\mathcal S}}
\newcommand{\cT}{{\mathcal T}}

\title[Constrained $L^p$ Approximation of Shape Tensors]
{Constrained $L^p$ Approximation of Shape Tensors and its Role for the Determination of Shape Gradients}

\author[Laura Hetzel and Gerhard Starke]{}


\thanks{The authors gratefully acknowledge support by the Deutsche Forschungsgemeinschaft in the Priority Program SPP 1962
``Non-smooth and Complementarity-based Distributed Parameter Systems: Simulation and Hierachical Optimization'' under the project number STA 402/13-2.}

\thanks{$^*$Corresponding author: Gerhard Starke}

\begin{document}
\maketitle

\centerline{\scshape
Laura Hetzel$^{{\href{mailto:laura.hetzel@uni-due.de}{\textrm{\Letter}}}1}$
and Gerhard Starke$^{{\href{mailto:gerhard.starke@uni-due.de}{\textrm{\Letter}}}*1}$}

\medskip

{\footnotesize
 \centerline{$^1$Fakult\"at f\"ur Mathematik, Universit\"at Duisburg-Essen, 45117 Essen, Germany}
} 

\bigskip

 \centerline{(Communicated by Handling Editor)}


\begin{abstract}
  This paper extends our earlier work \cite{Sta:24} on the $L^p$ approximation of the shape tensor by Laurain and Sturm \cite{LauStu:16}. In particular,
  it is shown that the weighted $L^p$ distance to an affine space of admissible {\em symmetric\/} shape tensors satisfying a divergence constraint provides the
  shape gradient with respect to the $L^{p^\ast}$-norm (where $1/p + 1/p^\ast = 1$) of the elastic strain associated with the shape deformation. This approach
  allows the combination of two ingredients which have already been used successfully in numerical shape optimization: (i) departing from the Hilbert space framework
  towards the Lipschitz topology approximated by $W^{1,p^\ast}$ with $p^\ast > 2$ and (ii) using the symmetric rather than the full gradient to define the norm.
  Similarly to \cite{Sta:24}, the $L^p$ distance measures the shape stationarity by means of the dual norm of the shape derivative with respect to the
  above-mentioned $L^{p^\ast}$-norm of the elastic strain. Moreover, the Lagrange multiplier for the momentum balance constraint constitute
  the steepest descent deformation with respect to this norm. The finite element realization of this approach is done using the weakly symmetric PEERS
  element and its three-dimensional counterpart, respectively. The resulting piecewise constant approximation for the Lagrange multiplier is reconstructed
  to a shape gradient in $W^{1,p^\ast}$ and used in an iterative procedure towards the optimal shape.
\end{abstract}


\section{Introduction}
\label{sec-introduction}

In the last decades, shape optimization has become a mature classical area of applied mathematics which is still open to new developments in its numerical
treatment. The survey paper by Allaire, Dapogny and Jouve \cite{AllDapJou:21} gave an excellent overview of the current status. Using the volume expression
of the shape derivative for the numerical approximation of shape optimization problems was systematically studied in \cite{HipPagSar:15} inspired by the earlier
work of Berggren \cite{Ber:10}. Besides being defined under weaker assumptions on the regularity of the underlying boundary value problems, the volume
representation of the shape derivative also proved to have advantages for the finite element approximation. The following contributions are also worth being
mentioned in this context: \cite{SchSieWel:16} developing Steklov-Poincar\'{e} metrics and \cite{EigStu:18} introducing variable metric algorithms based on reproducing
kernels. Shape optimization problems can also be viewed from the optimal control perspective, see e.g. \cite{HauSieUlb:21,ApeMatPfeRoe:18}. Nevertheless,
the problem of mesh degeneracy
in the course of an iterative method based on shape gradients remains, and two remedies were recently suggested: (i) restricting the admissible deformations
to those consistent with Hadamard's structure theorem \cite{EtlHerLoaWac:20} or in the spirit of a Riemannian metric \cite{HerLoa:24} and (ii) abandoning the
Hilbert space framework towards the Lipschitz topology for the definition of the shape gradient \cite{DecHerHin:22,DecHerHin:23a,DecHerHin:23b}.

Our approach is based on the shape tensor representation which, to the best of our knowledge, was studied first in the paper by Laurain and Sturm \cite{LauStu:16}.
It was extended in \cite{Lau:20} and \cite{LauLopNak:23} to shape optimization problems with rather general variational problems as constraints and specifically used
in the context of time-domain full waveform inversion in \cite{AlbLauYou:21} and high-temperature superconductivity in \cite{LauWinYou:21}.
Recently, in \cite{Sta:24}, the shape tensor was used as a starting point for measuring the distance of a given shape to stationarity and for the construction of shape
gradients in $W^{1,p^\ast}$ with $p^\ast \geq 2$.
To this end, an $L^p$ approximation problem (with $1/p + 1/p^\ast = 1$, $1 < p \leq p^\ast < \infty$) for the shape tensor subject to divergence constraints was derived.
In the present paper, we modify this approach by adding symmetry as constraint on the approximation space for the shape tensor. We will prove that the distance,
measured in a suitably weighted $L^p$-norm is equal to the dual norm of the shape derivative with respect to an $L^{p^\ast}$-norm associated with the linear elastic
strain of the deformation. This is of interest from a numerical point of view since elasticity-based formulations for the computation of shape gradients turned out
to have favorable properties with respect to mesh deformations (see e.g. \cite{EtlHerLoaWac:20} and \cite{BarWac:20}). Another issue where this contribution goes
beyond our earlier work \cite{Sta:24}  is the incorporation of constraints (like prescribed volume or perimeter).

The considered $L^p$ approximation problem for the shape tensor has some resemblance, at least for $p = 2$, to the constrained first-order system least squares
approach to linear elasticity, see e.g. \cite{AdlVas:14} and \cite{CaiSta:04}. The finite element approximation of the tensorial approximation is somewhat more
complicated than in \cite{Sta:24} due to the additionally imposed symmetry. Apparently, the simplest possible stable finite element combination for this purpose
is the PEERS (``planar elasticity element with reduced stress'') element \cite{ArnBreDou:84} and its three-dimensional variant (see \cite[Sect. 9.4.1]{BofBreFor:13}).
These approaches are based on matrix-valued (lowest-order) Raviart-Thomas spaces enriched by suitable divergence-free bubble functions. Higher-order
approximations would only lead to optimal convergence behavior if parametric Raviart-Thomas spaces \cite{BerSta:16} were used based on a
higher-order parametrization of the boundaries. We plan to explore this direction in the future but restrict ourselves to the lowest-order case in this contribution.
The shape deformation arises as a Lagrange multiplier for the divergence constraint which is approximated by piecewise constant functions for the above-mentioned
finite element combinations. Efficient procedures for the reconstruction of admissible deformations in $W^{1,p^\ast}$ from these discontinuous approximations are
available (see, e.g., \cite{Ste:91} or \cite{ErnVoh:15}).

The structure of this paper is as follows: We start in Section \ref{sec-shape_tensor} with the classical shape optimization problem constrained by the Poisson
equation with Dirichlet conditions and introduce the Laurain-Sturm shape tensor as well as the shape gradient with respect to an elastic strain norm. The
connection between the dual norm of the shape derivative and the constrained best approximation of the shape tensor is investigated in Section
\ref{sec-shape_tensor_best_approximation}.
Section \ref{sec-discretization} is concerned with the discretization of the best approximation problem in a suitable mixed finite element setting.
An iterative scheme of shape gradient type based on the reconstruction of an admissible deformation is described in Section \ref{sec-shape_gradient_iteration}.
Throughout the paper, a simple test example with a known optimal shape in form of a disk is used to illustrate the results numerically. In this section, our shape
gradient iteration is also tested for a more sophisticated example with a complicated optimal shape. For problems where the center of gravity and the orientation
is not known for the optimal shape, a combination of the shape gradient iteration with an adjustment of the rigid body modes is presented in Section
\ref{sec-rigid_body_modes}. Finally, Section \ref{sec-constraints} contains the important generalization of the shape tensor
approximation to problems with constraints like prescribed volume or surface area.

\section{Shape tensor formulation and shape gradients with respect to elastic strain norms}
\label{sec-shape_tensor}

We consider the minimization of a shape functional
\begin{equation}
  J (\Omega) = \int_\Omega j (u_\Omega) \: dx
  \label{eq:shape_functional}
\end{equation}
subject to the constraint
\begin{equation}
  ( \nabla u_\Omega , \nabla v )_{L^2 (\Omega)} = ( f , v )_{L^2 (\Omega)}
  \mbox{ for all } v \in H_0^1 (\Omega)
  \label{eq:bvp}
\end{equation}
for $u_\Omega \in H_0^1 (\Omega)$ among all shapes $\Omega \in \cS$, where
\begin{equation}
   \cS = \{ \Omega = (\id + \theta) \Omega_0 : \theta , (\id + \theta)^{-1} - \id \in W^{1,\infty} (\Omega;\R^d) \}
  \label{eq:set_of_shapes}
\end{equation}
with the identity mapping $\id : \R^d \rightarrow \R^d$ and a fixed reference domain $\Omega_0$, e.g. the unit disk in $\R^2$ or unit ball in $\R^3$.
The shape derivative with respect to deformations $\id + \chi$ applied to $\Omega$ exists under certain assumptions on $f$, $j ( \: \cdot \: )$ and $j^\prime ( \: \cdot \: )$
and can be expressed as
\begin{equation}
  \begin{split}
    J^\prime (\Omega) [ \chi ]
    & = \left( \left( (\div \: \chi) \: I - \left( \nabla \chi + (\nabla \chi)^T \right) \right) \nabla u_\Omega , \nabla y_\Omega \right) \\
    & \hspace{4cm} + \left( f \: \nabla y_\Omega , \chi \right) + \left( j (u_\Omega) , \div \: \chi \right) \: ,
  \end{split}
  \label{eq:shape_derivative_pre}
\end{equation}
where $y_\Omega \in H_0^1 (\Omega)$ solves the adjoint problem
\begin{equation}
  ( \nabla y_\Omega , \nabla z ) = - ( j^\prime (u_\Omega) , z ) \mbox{ for all } z \in H_0^1 (\Omega)
  \label{eq:adjoint_bvp}
\end{equation}
(cf. \cite[Prop. 4.5]{AllDapJou:21}).
In (\ref{eq:shape_derivative_pre}) and for the rest of this paper, $( \cdot , \cdot )$ denotes the duality pairing $( L^p (\Omega) , L^{p^\ast} (\Omega) )$ with
$1/p + 1/p^\ast = 1$ ($1 < p < \infty$). The corresponding duality pairing on the boundary $\partial \Omega$ will be denoted by $\langle \cdot , \cdot \rangle$.
Moreover, our notation is such that gradients of scalar functions (like $\nabla u_\Omega \in \R^d$) are considered as column
vectors while gradients of vector functions (like $\nabla \theta \in \R^{d \times d}$) are understood as each row being the gradient of a component of $\theta \in \R^d$.
It is easy to see that $J^\prime (\Omega)$ is a continuous linear functional on $W^{1,\infty} (\Omega;\R^d)$ if only the following assumptions are fulfilled:
$f \in L^2 (\Omega)$, $j (u) \in L^1 (\Omega)$ for all $u \in H_0^1 (\Omega)$, $j^\prime (u) \in H^{-1} (\Omega)$ for all $u \in H_0^1 (\Omega)$. Under stronger
assumptions to be made more precise in the next section, $J^\prime (\Omega)$ will even become continuous on the larger spaces $W^{1,p^\ast} (\Omega;\R^d)$
for $p^\ast < 1$.

With the above notation in mind, (\ref{eq:shape_derivative_pre}) can be written as
\begin{equation}
  J^\prime (\Omega) [ \chi ] = \left( K ( u_\Omega ,  y_\Omega ) , \nabla \chi \right)
  + \left( f \: \nabla y_\Omega , \chi \right) + \left( j (u_\Omega) , \div \: \chi \right)
  \label{eq:shape_derivative}
\end{equation}
with
\begin{equation}
  K ( u_\Omega , y_\Omega ) = \left( \nabla u_\Omega \cdot \nabla y_\Omega \right) \: I
  - \nabla y_\Omega \otimes \nabla u_\Omega - \nabla u_\Omega \otimes \nabla y_\Omega \: ,
  \label{eq:definition_K}
\end{equation}
where $\otimes$ stands for the outer product. The shape tensor $K ( u_\Omega , y_\Omega )$ was introduced and studied by Laurain and Sturm in \cite{LauStu:16}.
Note that, in general, only $K ( u_\Omega , y_\Omega ) \in L^p (\Omega ; \R^{d \times d})$ for $p$ only slightly larger than (and arbitrarily close to) 1 can be
guaranteed for the Sturm-Laurain shape tensor since
$\nabla u_\Omega$ and $\nabla y_\Omega$ may only be in $L^{2+\epsilon} ( \Omega ; \R^{d \times d} )$ without additional regularity assumptions.
This means that, to be on the safe side, $J^\prime (\Omega) [\chi]$ should only be applied to deformations $\chi \in W^{1,p^\ast} (\Omega;\R^d)$ for
correspondingly large $p^\ast$. While $p^\ast$ should ideally be chosen as large as possible, we will also see that the numerical construction of a shape gradient
will become more difficult for large values of $p^\ast$.

In the following derivation, we always assume $p^\ast \in [ 2 , \infty )$ (such that the dual exponent satisfies $p \in ( 1 , 2 ]$). For fixed $\Omega \in \cS$, the shape
derivative $J^\prime (\Omega)$ is a linear functional on the space of deformations $W^{1,p^\ast} (\Omega ; \R^d)$. Let us restrict ourselves to the subspace of
deformations
\begin{equation}
  \begin{split}
    \Theta^{p^\ast}_{\perp,{\rm RM}} & = \{ \theta \in W^{1,p^\ast} (\Omega ; \R^d) : ( \nabla \theta , \nabla \rho ) = 0 \mbox{ for all } \rho \in {\rm RM} (\R^d) \\
    & \hspace{4cm} \mbox{ and } ( \theta , e ) = 0 \mbox{ for all constant } e \in \R^d \} \: ,
  \end{split}
  \label{eq:deformation_orthogonal_RM}
\end{equation}
where ${\rm RM} (\R^d) = \{ \chi : \nabla \chi + (\nabla \chi)^T = 0 \}$ denotes the space of rigid body modes. This means that the considered deformations are
orthogonal with respect to the gradient inner product to all rotations and $L^2 (\Omega)$-orthogonal to translations. This is a perfectly admissible sense of
orthogonality with respect to the rigid body modes and will enable us to derive the intricate relations to symmetric approximations of the Sturm-Laurain shape
tensor $K ( u_\Omega,y_\Omega )$ in the next section. Note that restricting the deformations to $\Theta^{p^\ast}_{\perp,{\rm RM}}$, i.e., orthogonal to the
rigid body modes, means that the barycenter and the orientation in space is fixed. On the subspace $\Theta^{p^\ast}_{\perp,{\rm RM}}$,
$\| \varepsilon (\: \cdot \:) \|_{L^{p^\ast} (\Omega)}$ with $\varepsilon (\theta) = (\nabla \theta + (\nabla \theta)^T)/2$ constitutes a norm due to
Korn's inequality (see e.g. \cite{JiaKau:17}) and Poincare's inequality (see e.g. \cite[Theorem II.5.4]{Gal:11}). The shape derivative,
constituting an object in the dual space of the admissible deformations, may therefore be normed by
\begin{equation}
  \| J^\prime (\Omega) \|_{p^\ast;\perp,{\rm RM}}^\ast = \sup_{\chi \in \Theta^{p^\ast}_{\perp,{\rm RM}}}
  \frac{J^\prime (\Omega) [ \chi ]}{\| \varepsilon (\chi) \|_{L^{p^\ast} (\Omega)}} \: .
  \label{eq:dual_norm_elastic}
\end{equation}
A necessary condition for a deformation $\theta \in \Theta^{p^\ast}_{\perp,{\rm RM}}$ being a maximizer of the right-hand side in (\ref{eq:dual_norm_elastic})
is that
\begin{equation}
  \begin{split}
    D_\chi & \left. \left( \frac{J^\prime (\Omega) [\chi]}{\| \varepsilon (\chi) \|_{L^{p^\ast} (\Omega)}} \right) \right|_{\chi = \theta} [ \vartheta ] =
    \left. D_\chi \left( \frac{J^\prime (\Omega) [\chi]}{\left( \| \varepsilon (\chi) \|_{L^{p^\ast} (\Omega)}^{p^\ast} \right)^{1/p^\ast}} \right) \right|_{\chi = \theta} [ \vartheta ] \\
    & = \frac{J^\prime (\Omega) [\vartheta]}{\left( \| \varepsilon (\theta) \|_{L^{p^\ast} (\Omega)}^{p^\ast} \right)^{1/p^\ast}}
    - \frac{1}{p^\ast} \frac{J^\prime (\Omega) [\theta]}{\left( \| \varepsilon (\theta) \|_{L^{p^\ast} (\Omega)}^{p^\ast} \right)^{1/p^\ast+1}}
    \left( | \varepsilon (\theta) |^{p^\ast-2} \varepsilon (\theta) , \varepsilon (\vartheta) \right)
  \end{split}
  \label{eq:derivative_dual_norm_elastic}
\end{equation}
vanishes for all $\vartheta \in W^{1,p^\ast} (\Omega;\R^d)$. This is equivalent to
\begin{equation}
  \frac{1}{p^\ast} \frac{J^\prime (\Omega) [\theta]}{\| \varepsilon (\theta) \|_{L^{p^\ast} (\Omega)}^{p^\ast}}
  \left( | \varepsilon (\theta) |^{p^\ast-2} \varepsilon (\theta) , \varepsilon (\vartheta) \right)
  = J^\prime (\Omega) [\vartheta] \mbox{ for all } \vartheta \in W^{1,p^\ast} (\Omega;\R^d)
  \label{eq:p_dual_norm_elastic_variational_pre}
\end{equation}
or, since (\ref{eq:p_dual_norm_elastic_variational}) is independent of the scaling of $\theta$, simply to
\begin{equation}
  \left( | \varepsilon (\theta) |^{p^\ast-2} \varepsilon (\theta) , \varepsilon (\vartheta) \right) = J^\prime (\Omega) [\vartheta]
  \mbox{ for all } \vartheta \in W^{1,p^\ast} (\Omega;\R^d) \: .
  \label{eq:p_dual_norm_elastic_variational}
\end{equation}
If we choose $\vartheta \in {\rm RM} (\R^d)$ in (\ref{eq:p_dual_norm_elastic_variational}), we get that $J^\prime (\Omega) [\vartheta] = 0$ for all
$\vartheta \in {\rm RM} (\R^d)$. Using (\ref{eq:shape_derivative}) this leads to the compatibility condition
\begin{equation}
  ( f \: \nabla y_\Omega , \rho ) = 0 \mbox{ for all } \rho \in {\rm RM} (\R^d) \: .
  \label{eq:compatibility_condition_RM}
\end{equation}
This condition means that the shape derivative vanishes in direction of all rigid body modes which has the geometrical interpretation that the barycenter
and the orientation of $\Omega$ are already known to be optimal.

Using the fact that $K ( u_\Omega , y_\Omega )$ is symmetric, i.e., $K ( u_\Omega , y_\Omega ) \in L^p (\Omega ; \S^{d \times d})$ with
$\S^{d \times d} \subset \R^{d \times d}$ denoting the subspace of symmetric matrices, we can rewrite (\ref{eq:shape_derivative}) as
\begin{equation}
  J^\prime (\Omega) [ \vartheta ] = \left( K ( u_\Omega ,  y_\Omega ) , \varepsilon (\vartheta) \right)
  + \left( f \: \nabla y_\Omega , \vartheta \right) + \left( j (u_\Omega) , \div \: \vartheta \right) \:
  \label{eq:shape_derivative_symmetric}
\end{equation}
provided that $\vartheta \in L^{p^\ast} (\Omega ; \R^d)$.
The variational equation (\ref{eq:p_dual_norm_elastic_variational}) for $\theta \in \Theta_{\perp,{\rm RM}}^{p^\ast}$ can therefore be rewritten as
\begin{equation}
  \left( | \varepsilon (\theta) |^{p^\ast-2} \varepsilon (\theta) - K ( u_\Omega , y _\Omega ) , \varepsilon (\vartheta) \right)
  = \left( f \: \nabla y_\Omega , \vartheta \right) + \left( j (u_\Omega) , \div \: \vartheta \right)
  \mbox{ for all } \vartheta \in \Theta_{\perp,{\rm RM}}^{p^\ast} \: .
  \label{eq:shape_tensor_variational}
\end{equation}
If we define $S := K ( u_\Omega , y_\Omega ) - | \varepsilon (\theta) |^{p^\ast-2} \varepsilon (\theta)$, a straightforward calculation leads to
\begin{equation}
  \begin{split}
    \| S - K ( u_\Omega , y_\Omega ) \|_{L^p (\Omega)}^p & = \int_\Omega | \varepsilon (\theta) |^{(p^\ast-1) p} dx
    = \int_\Omega | \varepsilon (\theta) |^{(p^\ast-1) / (1 - 1/p^\ast)} dx \\
    & = \int_\Omega | \varepsilon (\theta) |^{p^\ast} dx = \| \varepsilon (\theta) \|_{L^{p^\ast} (\Omega)}^{p^\ast} \: .
  \end{split}
  \label{eq:relation_p_pstar}
\end{equation}
This means that, under the assumption that $K ( u_\Omega , y_\Omega ) \in L^p (\Omega ; \S^{d \times d})$, we also have $S \in L^p (\Omega ; \S^{d \times d})$.
Moreover, by definition and using (\ref{eq:shape_tensor_variational}), $S$ satisfies
\begin{equation}
  ( S , \varepsilon (\vartheta) ) + \left( f \: \nabla y_\Omega , \vartheta \right) + \left( j (u_\Omega) , \div \: \vartheta \right) = 0
  \mbox{ for all } \vartheta \in \Theta_{\perp,{\rm RM}}^{p^\ast} \: .
\end{equation}
Let us define the affine subspace
\begin{equation}
  \begin{split}
    \widetilde{\Sigma}^{p,0}_{f,j,{\rm sym}}
    & = \{ T \in L^p (\Omega ; \S^{d \times d}) : ( T , \varepsilon (\vartheta) ) + ( f \: \nabla y_\Omega , \vartheta ) + ( j ( u_\Omega) , \div \: \vartheta ) = 0 \\
    & \hspace{7cm} \mbox{ for all } \vartheta \in \Theta_{\perp,{\rm RM}}^{p^\ast} \}
  \end{split}
  \label{eq:constraint_set_p}
\end{equation}
and note that $S \in \widetilde{\Sigma}^{p,0}_{f,j,{\rm sym}}$. For all $T \in \widetilde{\Sigma}^{p,0}_{f,j,{\rm sym}}$ and for all
$\vartheta \in \Theta_{\perp,{\rm RM}}^{p^\ast}$, we have
\begin{equation}
  J^\prime (\Omega) [\vartheta] = ( K ( u_\Omega , y_\Omega ) - T , \varepsilon (\vartheta) )
  \label{eq:relation_shape_derivative_tensor}
\end{equation}
and therefore also
\begin{equation}
  \frac{J^\prime (\Omega) [\vartheta]}{\| \varepsilon (\vartheta) \|_{L^{p^\ast} (\Omega)}}
  = \frac{( K ( u_\Omega , y_\Omega ) - T , \varepsilon (\vartheta) )}{\| \varepsilon (\vartheta) \|_{L^{p^\ast} (\Omega)}}
  \leq \| K ( u_\Omega , y_\Omega ) - T \|_{L^p (\Omega)}
  \label{eq:relation_shape_derivative_tensor_bound}
\end{equation}
by H\"older's inequality which implies that
\begin{equation}
  \| J^\prime (\Omega) \|_{p^\ast;\perp,{\rm RM}}^\ast \leq \| K ( u_\Omega , y_\Omega ) - T \|_{L^p (\Omega)}
  \mbox{ for all } T \in \widetilde{\Sigma}^{p,0}_{f,j,{\rm sym}} \: .
  \label{eq:upper_bound_tensor_approximation}
\end{equation}
On the other hand, we have
\begin{equation}
  \begin{split}
    \| J^\prime (\Omega) \|_{p^\ast;\perp,{\rm RM}}^\ast
    & = \frac{J^\prime (\Omega) [\theta]}{\| \varepsilon (\theta) \|_{L^{p^\ast} (\Omega)}}
    = \frac{( K ( u_\Omega , y_\Omega ) - S , \varepsilon (\theta) )}{\| \varepsilon (\theta) \|_{L^{p^\ast} (\Omega)}} \\
    & = \| \varepsilon (\theta) \|_{L^{p^\ast} (\Omega)}^{p^\ast-1} = \| \varepsilon (\theta) \|_{L^{p^\ast} (\Omega)}^{p^\ast/p} 
    = \| K ( u_\Omega , y_\Omega ) - S \|_{L^p (\Omega)}
  \end{split}
  \label{eq:equality_tensor_approximation}
\end{equation}
due to the definition of $S$ and (\ref{eq:relation_p_pstar}). We have thus shown that
\begin{equation}
  \| J^\prime (\Omega) \|_{p^\ast;\perp,{\rm RM}}^\ast = \min_{T \in \widetilde{\Sigma}^{p,0}_{f,j,{\rm sym}}} \| T - K ( u_\Omega , y_\Omega ) \|_{L^p (\Omega)} \: .
  \label{eq:dual_norm_as_best_approximation_symmetric_pre}
\end{equation}
In the next section, we will exploit this relation by investigating the best approximation problem on the right hand side in
(\ref{eq:dual_norm_as_best_approximation_symmetric_pre}) in detail.

We close this section by putting these results into perspective with respect to our recent work \cite{Sta:24}. There, we considered the following dual norm of the
shape derivative:
\begin{equation}
  \| J^\prime (\Omega) \|_{p^\ast;\perp}^\ast = \sup_{\chi \in \Theta_\perp^{p^\ast}}
  \frac{J^\prime (\Omega) [ \chi ]}{\| \nabla \chi \|_{L^{p^\ast} (\Omega)}} \: ,
  \label{eq:dual_norm}
\end{equation}
where $\Theta_\perp^{p^\ast} = \{ \chi \in W^{1,p^\ast} (\Omega ; \R^d) : ( \chi , e ) = 0 \mbox{ for all constant } e \in \R^d \}$ (since only the constants
are in the null space of the gradient operator). In \cite[Theorem 3.3]{Sta:24}, it is proved that
\begin{equation}
  \| J^\prime (\Omega) \|_{p^\ast;\perp}^\ast = \min_{T \in \Sigma^{p,0}_{f,j}} \| T - K ( u_\Omega , y_\Omega ) \|_{L^p (\Omega)}
  \label{eq:dual_norm_as_best_approximation}
\end{equation}
holds, where
\begin{equation}
  \begin{split}
    \Sigma^{p,0}_{f,j} & = \{ T \in L^p (\Omega;\R^{d \times d}) : ( \div \: T , \chi ) =  ( f \: \nabla y_\Omega , \chi ) + ( j ( u_\Omega) , \div \: \chi ) \\
    & \hspace{4cm} \mbox{ and } \langle T \cdot n |_{\partial \Omega} , \chi |_{\partial \Omega} \rangle = 0 \mbox{ for all } \chi \in \Theta_\perp^{p^\ast} \}
  \end{split}
  \label{eq:constraints_old}
\end{equation}
(again assuming that $( f \: \nabla y_\Omega , e ) = 0$ holds for all constant $e \in \R^d$). In short, $\Sigma^{p,0}_{f,j}$ is the modification of the space
$\widetilde{\Sigma}^{p,0}_{f,j,{\rm sym}}$ above, where the symmetry condition is removed and integration by parts was applied.

To derive the relation between the $L^p (\Omega)$ best approximation of the shape tensor and the dual norm of the shape derivative again here for the
symmetric case and the norm associated with linear elastic strain is, in our opinion, not just a generalization for its own sake. It is motivated by the recent
positive experiences with using elasticity type norms in shape optimization (see e.g. \cite{EtlHerLoaWac:20,BarWac:20}) and the possible benefits of
combining this with the departure from the Hilbert space setting by computing shape gradients with respect to $W^{1,p^\ast} (\Omega ; \R^d)$ for $p^\ast > 2$.

\section{Constrained $L^p$ best approximation of shape tensors and its relation to $W^{1,p^\ast}$ shape gradients}
\label{sec-shape_tensor_best_approximation}

In this section, we study the best approximation problem on the right-hand side in (\ref{eq:dual_norm_as_best_approximation_symmetric_pre}) in more detail.
The following result is similar to \cite[Theorem 2.1]{Sta:24}, adapted to the different setting described above.

\begin{theorem}
  Let $p \in ( 1 , 2 ]$ and assume that $K ( u_\Omega , y_\Omega ) \in L^p ( \Omega ; \S^{d \times d} )$.
  If the compatibility condition $( f \: \nabla y_\Omega , \rho ) = 0$ for all $\rho \in {\rm RM} (\R^d)$ is fulfilled, then
  there is a uniquely determined $L^p (\Omega)$-best approximation $S \in \widetilde{\Sigma}^{p,0}_{f,j,{\rm sym}}$ to $K ( u_\Omega , y_\Omega )$, i.e.,
  \begin{equation}
    \| S - K ( u_\Omega , y_\Omega ) \|_{L^p (\Omega)} = \inf_{T \in \widetilde{\Sigma}^{p,0}_{f,j,{\rm sym}}} \| T - K ( u_\Omega , y_\Omega ) \|_{L^p (\Omega)} \: .
    \label{eq:best_approximation_S}
  \end{equation}
  \label{theorem-unique_S}
\end{theorem}

\begin{proof}
  We start by showing that the admissible affine subspace $\widetilde{\Sigma}^{p,0}_{f,j,{\rm sym}}$ defined by (\ref{eq:constraint_set_p}) is not empty. To this end, write
  $\bar{S} = \varepsilon (\psi)$ where $\psi \in H^1 (\Omega ; \R^d)$ solves the variational problem
  \begin{equation}
      ( \varepsilon (\psi) , \varepsilon (\chi) ) = - ( f \: \nabla y_\Omega , \chi ) - ( j (u_\Omega) , \div \: \chi ) \mbox{ for all } \chi \in H^1 (\Omega;\R^d) \: .
    \label{eq:auxiliary_Neumann_problem}
  \end{equation}
  Due to the compatibility condition, (\ref{eq:auxiliary_Neumann_problem}) has a solution which is unique up to arbitrary rigid body modes $\rho \in {\rm RM} (\R^d)$.
  We see that $\bar{S} \in \widetilde{\Sigma}^{2,0}_{f,j,{\rm sym}} \subseteq \widetilde{\Sigma}^{p,0}_{f,j,{\rm sym}}$ for $1 < p \leq 2$.

  The existence of a best approximation satisfying (\ref{eq:best_approximation_S}) was already proved in the previous section and also follows from the fact that
  $\widetilde{\Sigma}^{p,0}_{f,j,{\rm RM}}$ is a closed affine subspace of the reflexive Banach space $L^p (\Omega ; \R^{d \times d})$, $p \in ( 1 , 2 ]$
  (see \cite[Theorem 3.3.14]{AtkHan:09}). Since, again for $p \in ( 1 , 2 ]$,
  $L^p (\Omega ; \R^{d \times d})$ is a strictly normed space, the best approximation is unique (see \cite[Theorem 3.3.21]{AtkHan:09}).
\end{proof}

We will now restrict our affine subspace of admissible shape tensor approximations to
\begin{equation}
  \begin{split}
    \Sigma_{f,j,{\rm sym}}^{p,0} & = \{ T \in L^p (\Omega;\S^{d \times d}) : \div \: T \in L^p (\Omega;\R^d) \: , \: \left. T \cdot n \right|_{\partial \Omega} = 0 \mbox{ and } \\
    & \hspace{1cm}
    ( \div \: T , \vartheta ) = ( f \: \nabla y_\Omega , \vartheta ) + ( j (u_\Omega) , \div \: \vartheta ) \mbox{ for all } \vartheta \in \Theta_{\perp,{\rm RM}}^{p^\ast} \} \: .
  \end{split}
  \label{eq:constraint_subspace_p}
\end{equation}
Through integration by parts, we see immediately that $\Sigma_{f,j,{\rm sym}}^{p,0} \subset \widetilde{\Sigma}_{f,j,{\rm sym}}^{p,0}$ holds. In order to make sure
that $\Sigma_{f,j,{\rm sym}}^{p,0}$ is not empty, we need stronger assumptions on $f$ and $j$. It suffices to assume $f \in L^{2 p /(2 - p)} (\Omega)$
($f \in L^\infty (\Omega)$ if $p = 2$) and $j (u_\Omega) \in L^p (\Omega)$. Under these conditions, Theorem \ref{theorem-unique_S} is also valid if
$\widetilde{\Sigma}^{p,0}_{f,j,{\rm sym}}$ is replaced by $\Sigma^{p,0}_{f,j,{\rm sym}}$ and we will consider this version of (\ref{eq:best_approximation_S}) for
the rest of this paper. The determination of the minimum in (\ref{eq:best_approximation_S}) constitutes a convex minimization problem in the space
\begin{equation}
  \begin{split}
    \Sigma^{p,0} & = \{ T \in L^p (\Omega;\R^{d \times d}) : \div \: T \in L^p (\Omega;\R^d) \mbox{ and }
    \langle \left. T \cdot n \right|_{\partial \Omega} , \left. \chi \right|_{\partial \Omega} \rangle = 0 \\
    & \hspace{4cm} \mbox{ for all } \chi \in W^{1,p^\ast} (\Omega ; \R^d) \}
  \end{split}
\end{equation}
subject to the divergence constraint appearing in (\ref{eq:constraint_subspace_p}) and the additional constraint that $T$ be symmetric. For the symmetry constraint,
the space
\begin{equation}
  \Xi^p = \{ \rho \in L^p (\Omega;\R^{d \times d}) : \rho + \rho^T = 0 \: , \: ( \rho , E ) = 0 \mbox{ for all constant } E \in \R^{d \times d} \}
  \label{eq:rotation_space}
\end{equation}
will be useful. We will now show that the corresponding KKT conditions for the solution $S \in \Sigma^{p,0}$ of the minimization problem are given by
\begin{equation}
  \begin{split}
    ( | S - K (u_\Omega,y_\Omega) |^{p-2} (S - K (u_\Omega,y_\Omega)) , T ) + ( \div \: T , \theta ) + ( \as \: T , \omega ) & = 0 \: , \\
    ( \div \: S , \chi ) - ( f \: \nabla y_\Omega , \chi ) - ( j (u_\Omega) , \div \: \chi ) & = 0 \: , \\
    ( \as \: S , \gamma ) & = 0
  \end{split}
  \label{eq:KKT}
\end{equation}
for all $T \in \Sigma^{p,0}$, $\chi \in \Theta^{p^\ast}_{\perp,{\rm RM}}$ and $\gamma \in \Xi^{p^\ast}$. Here, $\theta \in \Theta^{p^\ast}_{\perp,{\rm RM}}$ and
$\omega \in \Xi^{p^\ast}$ are Lagrange multipliers associated with the divergence and symmetry constraints, respectively.

\begin{remark}
  In the proof of Theorem \ref{thrm-saddle_point_problem_p} below, we will make use of a Helmholtz decomposition of the form
  \begin{equation}
    L^{p^\ast} (\Omega;\R^{d \times d}) = \nabla W^{1,p^\ast} (\Omega;\R^d) \oplus H_{p^\ast} (\Omega;\R^{d \times d}) \: ,
    \label{eq:Helmholtz}
  \end{equation}
  where $H_{p^\ast} (\Omega;\R^{d \times d})$ denotes the subspace of functions where the divergence and the boundary trace vanishes. The validity
  of a Helmholtz decomposition (\ref{eq:Helmholtz}) depends on the validity of a certain Neumann boundary value problem for the Poisson equation
  (see \cite[Sect. 11]{FabMenMit:98}). For general Lipschitz domains this is guaranteed for $3/2 \leq p^\ast \leq 3$ while it is known to hold for all
  $p^\ast \in (1,\infty)$ in the case of a domain with $C^2$ boundary (cf. \cite{FujMor:77}, \cite[Theorem III.3.3]{Gal:11}).
  \label{rmrk-Helmholtz}
\end{remark}

\begin{theorem}
  Let $p \in ( 1 , 2 ]$ be such that $L^{p^\ast} (\Omega;\R^{d \times d})$ possesses a Helmholtz decomposition as stated in Remark \ref{rmrk-Helmholtz}
  and assume that $K (u_\Omega,y_\Omega) \in L^p (\Omega;\S^{d \times d})$ as well as
  $f \in L^\infty (\Omega)$ and $j (u_\Omega) \in L^2 (\Omega)$ hold. Moreover, assume that the compatibility condition
  $( f \: \nabla y_\Omega , \rho ) = 0$ for all $\rho \in {\rm RM} (\R^d)$ is fulfilled. Then, (\ref{eq:KKT}) has a unique solution
  \[
    (S,\theta,\omega) \in \Sigma^{p,0} \times \Theta_{\perp,{\rm RM}}^{p^\ast} \times \Xi^{p^\ast} \;\;\; ( 1/p + 1/p^\ast = 1 ) \: . 
  \]
  $S \in \Sigma^{p,0}$ is the unique minimizer of (\ref{eq:best_approximation_S}).
  \label{thrm-saddle_point_problem_p}
\end{theorem}

Before we can start with the proof, some preparatory results are needed. To this end, we will use the following subspaces of $\Sigma^{p,0}$:
\begin{equation}
  \begin{split}
    \Sigma^{p,0}_{\rm sol} & := \{ T \in \Sigma^{p,0} : \div \: T = 0 \} \: , \\
    \Sigma^{p,0}_{\rm sol,{\rm RM}} & := \{ T \in \Sigma^{p,0} : ( \div \: T , \rho ) = 0 \mbox{ for all } \rho \in \mbox{RM} (\R^d) \} \: .
  \end{split}
  \label{eq:solenoidal_subspaces}
\end{equation}
Obviously, $\Sigma^{p,0}_{\rm sol} \subset \Sigma^{p,0}_{\rm sol,{\rm RM}}$ and, under the compatibility condition in Theorem 3.1, $S \in \Sigma^{p,0}_{\rm sol,RM}$
holds due to the divergence constraint in (\ref{eq:KKT}).

\begin{lemma}
  For all $T \in \Sigma^{p,0}_{\rm sol,RM}$, $\as \: T \in \Xi^p$ holds.
\end{lemma}

\begin{proof}
  For all
  \begin{equation}
    \rho \in \mbox{span} \{ \begin{pmatrix} \;\;\;x_2 \\ - x_1 \\ \;0 \end{pmatrix} , \begin{pmatrix} \;0 \\ \;\;\;x_3 \\ - x_2 \end{pmatrix} ,
                                        \begin{pmatrix} - x_3 \\ \;0 \\ \;\;\;x_1 \end{pmatrix} \} \subset {\rm RM} (\R^d) \: ,
  \end{equation}
  integration by parts implies $0 = - ( \div \: S , \rho ) = ( S , \nabla \rho ) = ( \as \: S , \nabla \rho )$. Since all constant skew-symmetric $E \in \R^{d \times d}$ can
  be expressed as $E = \nabla \rho$ and (\ref{eq:as_average_zero}), we obtain
  \begin{equation}
    ( \as \: S , E ) = 0 \mbox{ for all constant skew-symmetric } E \in \R^{d \times d} \: .
    \label{eq:as_average_zero}
  \end{equation}
  Moreover, (\ref{eq:as_average_zero}) holds trivially for all symmetric $E \in \R^{d \times d}$ which completes the proof.
\end{proof}

A crucial ingredient in the proof of Theorem \ref{thrm-saddle_point_problem_p} is the surjectivity of the mapping
\[
  \as : \Sigma^{p,0}_{\rm sol} \rightarrow \Xi^p \: .
\]
For $p = 2$, this result is part of the proof of \cite[Proposition 9.3.2]{BofBreFor:13}. We give a self-contained proof for the result in the general case $p \in ( 1 , \infty )$
which we state as follows.

\begin{lemma}
  For each $\xi \in \Xi^p$, $1 < p < \infty$, there is a $T \in \Sigma^{p,0}_{\rm sol}$ with $\as \: T = \xi$.   
  \label{lmm-as_surjective}
\end{lemma}

\begin{proof}
  We present the proof for three dimensions; the case $d = 2$ is similar but less technical. Given $\xi \in \Xi^p$, we construct $T \in \Sigma^{p,0}_{\rm sol}$
  in the form
  \begin{equation}
    T = \curl \: \psi \mbox{ with } \psi \in W^{1,p} (\Omega;\R^{d \times d}) \: , \: \left. \curl \: \psi \cdot n \right|_{\partial \Omega} = 0 \: .
    \label{eq:sigma_as_curl}
  \end{equation}
  $\psi$ will be made up of two parts $\psi = \psi^{(1)} + \psi^{(2)}$. We start with $\psi^{(1)} \in W_0^{1,p} (\Omega ; \R^{d \times d})$ such that
  \begin{equation}
    \div \: (\psi^{(1)})^T + 2 \begin{pmatrix} \xi_{23} \\ \xi_{31} \\ \xi_{12} \end{pmatrix} = 0 \: .
  \end{equation}
  This is possible due to \cite[Theorem III.3.3]{Gal:11}. We have
  \begin{equation}
    \begin{split}
      \as & ( \curl \: \psi^{(1)} ) = \as \begin{pmatrix}
        \partial_2 \psi_{13}^{(1)} - \partial_3 \psi_{12}^{(1)}  & \partial_3 \psi_{11}^{(1)} - \partial_1 \psi_{13}^{(1)}  & \partial_1 \psi_{12}^{(1)} - \partial_2 \psi_{11}^{(1)} \\
        \partial_2 \psi_{23}^{(1)} - \partial_3 \psi_{22}^{(1)}  & \partial_3 \psi_{21}^{(1)} - \partial_1 \psi_{23}^{(1)}  & \partial_1 \psi_{22}^{(1)} - \partial_2 \psi_{21}^{(1)} \\
        \partial_2 \psi_{33}^{(1)} - \partial_3 \psi_{32}^{(1)}  & \partial_3 \psi_{31}^{(1)} - \partial_1 \psi_{33}^{(1)}  & \partial_1 \psi_{32}^{(1)} - \partial_2 \psi_{31}^{(1)}
      \end{pmatrix} \\
      = & \frac{1}{2} \left( \begin{pmatrix}
        0 & \partial_3 (\psi_{11}^{(1)} + \psi_{22}^{(1)}) & \mbox{antisymm.} \\
        \mbox{antisymm.} & 0 & \partial_1 (\psi_{22}^{(1)} + \psi_{33}^{(1)}) \\
        \partial_2 (\psi_{11}^{(1)} + \psi_{33}^{(1)}) & \mbox{antisymm.} & 0
      \end{pmatrix} \right. \\
      & \;\;\;\;\;\; \left. - \begin{pmatrix}
        0 & \partial_1 \psi_{13}^{(1)} + \partial_2 \psi_{23}^{(1)} & \mbox{antisymm.} \\
        \mbox{antisymm.} & 0 & \partial_2 \psi_{21}^{(1)} + \partial_3 \psi_{31}^{(1)} \\
        \partial_1 \psi_{12}^{(1)} + \partial_3 \psi_{32}^{(1)} & \mbox{antisymm.} & 0
      \end{pmatrix} \right) \\
      = & \frac{1}{2} \begin{pmatrix}
        0 & \partial_3 (\tr \: \psi^{(1)}) - \div \: \psi_{\cdot,3}^{(1)} & \div \: \psi_{\cdot,2}^{(1)} - \partial_2 (\tr \: \psi^{(1)}) \\
        \div \: \psi_{\cdot,3}^{(1)} - \partial_3 (\tr \: \psi^{(1)}) & 0 & \partial_1 (\tr \: \psi^{(1)}) - \div \: \psi_{\cdot,1}^{(1)} \\
        \partial_2 (\tr \: \psi^{(1)}) - \div \: \psi_{\cdot,2}^{(1)} & \div \: \psi_{\cdot,1}^{(1)} - \partial_1 (\tr \: \psi^{(1)}) & 0
      \end{pmatrix} \\
      = & \frac{1}{2} \begin{pmatrix}
        0 & \;\partial_3 (\tr \: \psi^{(1)}) & \!- \partial_2 (\tr \: \psi^{(1)}) \\
        \!- \partial_3 (\tr \: \psi^{(1)}) & 0 & \;\partial_1 (\tr \: \psi^{(1)}) \\
        \;\partial_2 (\tr \: \psi^{(1)}) & \!- \partial_1 (\tr \: \psi^{(1)}) & 0
      \end{pmatrix}
      + \begin{pmatrix}
        0 & \;\xi_{12} & \!- \xi_{31} \\
        \!- \xi_{12} & 0 & \;\xi_{23} \\
        \;\xi_{31} & \!- \xi_{23} & 0
      \end{pmatrix} ,
    \end{split}
  \end{equation}
  which means that $\psi^{(2)} \in W^{1,p} (\Omega;\R^{d \times d})$ needs to satisfy
  \begin{equation}
    \begin{split}
      \as ( \curl \: \psi^{(2)} ) + \frac{1}{2} \begin{pmatrix}
        0 & \;\;\partial_3 (\tr \: \psi^{(1)}) & - \partial_2 (\tr \: \psi^{(1)}) \\
        - \partial_3 (\tr \: \psi^{(1)}) & 0 & \;\;\partial_1 (\tr \: \psi^{(1)}) \\
        \;\;\partial_2 (\tr \: \psi^{(1)}) & - \partial_1 (\tr \: \psi^{(1)}) & 0
      \end{pmatrix} & = 0 \\
      \mbox{ and } \left. \curl \: \psi^{(2)} \cdot n \right|_{\partial \Omega} & = 0 \: .
    \end{split}
  \end{equation}
  Obviously, $\psi^{(2)} = \frac{1}{2} (\tr \: \psi^{(1)}) \: I$ does this job.
\end{proof}

{\em Proof of Theorem 3.3.}
From Theorem \ref{theorem-unique_S} we deduce that there is a unique $S \in \Sigma^{p,0}$ which minimizes (\ref{eq:best_approximation_S}).
Let $\bar{S} \in \Sigma^{p,0}$ be such that the constraints in (\ref{eq:best_approximation_S}) are satisfied (like the one constructed in the proof of Theorem
\ref{theorem-unique_S}). The minimization problem can then be written as finding $S = \bar{S} + S^\circ$ such that $S^\circ \in \Sigma^{p,0}_{\rm sol}$ solves
\begin{equation}
  \begin{split}
    \| \bar{S} + T^\circ - K ( u_\Omega , y_\Omega ) \|_{L^p (\Omega)} \rightarrow \min! \;\; & \mbox{ among all } \;\; T^\circ \in \Sigma^{p,0}_{\rm sol} \\
    & \mbox{ which satisfy } \;\; \as \: T^\circ = 0 \: .
  \end{split}
  \label{eq:constraint_as}
\end{equation}
Due to the surjectivity, shown in Lemma \ref{lmm-as_surjective}, the theory of Lagrange multipliers in constrained optimization (see e.g. \cite[Section 1.3]{ItoKun:08})
implies the existence of an $\widetilde{\omega} \in L^{p^\ast} (\Omega;\R^{d \times d})$, satisfying $\widetilde{\omega} + \widetilde{\omega}^T = 0$, such that
\begin{equation}
  ( | S - K (u_\Omega,y_\Omega) |^{p-2} (S - K (u_\Omega,y_\Omega)) , T ) + ( \as \: T , \widetilde{\omega}) = 0
  \label{eq:KKT_as}
\end{equation}
holds for all $T \in \Sigma^{p,0}_{\rm sol}$. This implies that
\begin{equation}
  ( | S - K (u_\Omega,y_\Omega) |^{p-2} (S - K (u_\Omega,y_\Omega)) + \widetilde{\omega} , T ) = 0
\end{equation}
holds for all $T \in \Sigma^{p,0}_{\rm sol}$. Therefore, the Helmholtz decomposition (for each row) in
$L^{p^\ast} (\Omega;\R^{d \times d})$ (see Remark \ref{rmrk-Helmholtz} above) states that
\begin{equation}
  | S - K (u_\Omega,y_\Omega) |^{p-2} (S - K (u_\Omega,y_\Omega)) + \widetilde{\omega} = \nabla \tilde{\theta}
\end{equation}
with some $\tilde{\theta} \in W^{1,p^\ast} (\Omega;\R^d)$. Finally, we construct a rotation $\rho \in {\rm RM} (\R^d)$ in such a way that
\begin{equation}
  ( \nabla \tilde{\theta} - \nabla \rho , \nabla \gamma ) = 0 \mbox{ for all } \gamma \in {\rm RM} (\R^d)
\end{equation}
holds. This is a nonsingular linear system of dimension $d (d-1) / 2$ (only the rotations are involved). Setting $\theta = \tilde{\theta} - \rho$ and
$\omega = \widetilde{\omega} - \nabla \rho$ leads to
\begin{equation}
  | S - K (u_\Omega,y_\Omega) |^{p-2} (S - K (u_\Omega,y_\Omega)) + \omega = \nabla \theta
  \label{eq:multiplier_gradient}
\end{equation}
with $( \nabla \theta , \nabla \gamma ) = 0$ for all $\gamma \in {\rm RM} (\R^d)$. Finally $\theta \in \Theta_{\perp,{\rm RM}}^{p^\ast}$ can be achieved by
adding a suitable constant vector which leaves (\ref{eq:multiplier_gradient}) unaltered. Moreover, we also get
$\omega \in \Xi^{p^\ast}$ since for all constant skew-symmetric $E \in \R^{d \times d}$, (\ref{eq:multiplier_gradient}) gives
\begin{equation}
  ( \omega , E ) = ( \nabla \theta , E ) - ( | S - K (u_\Omega,y_\Omega) |^{p-2} (S - K (u_\Omega,y_\Omega)) , E ) = ( \nabla \theta , \nabla \hat{\rho} )
  \label{eq:multiplier_orthogonality}
\end{equation}
due to the symmetry of $S - K (u_\Omega,y_\Omega)$ and with some $\hat{\rho} \in {\rm RM} (\R^d)$. Due to $\theta \in \Theta_{\perp,{\rm RM}}^{p^\ast}$,
the term in (\ref{eq:multiplier_orthogonality}) is zero. Multiplication of (\ref{eq:multiplier_gradient}) with test functions $T \in \Sigma^{p,0}$ and integration by
parts leads to the first equation in (\ref{eq:KKT}).

Uniqueness of the Lagrange multipliers follows from the fact that (\ref{eq:multiplier_gradient}) implies
\begin{equation}
  \varepsilon (\theta) = | S - K (u_\Omega,y_\Omega) |^{p-2} (S - K (u_\Omega,y_\Omega)) \: ,
  \label{eq:multiplier_symmetric_gradient}
\end{equation}
which determines $\theta$ up to elements in ${\rm RM} (\R^d)$ which are, in turn, fixed by the constraints in the definition of $\Theta_{\perp,{\rm RM}}^{p^\ast}$
in (\ref{eq:deformation_orthogonal_RM}). Moreover, (\ref{eq:multiplier_gradient}) also implies that $\omega = (\nabla \theta - \nabla \theta^T)/2$ holds which now
uniquely determines $\omega \in \Xi^{p^\ast}$. $\Box$

\begin{theorem}
  The dual norm of the shape derivative (\ref{eq:dual_norm_elastic}) and the best approximation (\ref{eq:best_approximation_S}) are connected by
  \begin{equation}
    \| J^\prime (\Omega) \|_{p^\ast;\perp,{\rm RM}}^\ast = \inf_{T \in \Sigma^{p,0}_{f,j,{\rm sym}}} \| T - K ( u_\Omega , y_\Omega ) \|_{L^p (\Omega)} \: .
    \label{eq:dual_norm_as_best_approximation_symmetric}
  \end{equation}
  Moreover, the Lagrange multiplier $\theta \in \Theta_{\perp,{\rm RM}}^{p^\ast}$ from (\ref{eq:KKT}) is the direction of steepest descent with respect to the dual
  norm of the shape derivative on the left in (\ref{eq:dual_norm_as_best_approximation_symmetric}) in the sense that
  \begin{equation}
    \frac{J^\prime (\Omega) [\theta]}{\| \varepsilon (\theta) \|_{L^{p^\ast} (\Omega)}}
    = \inf_{\chi \in \Theta_{\perp,{\rm RM}}} \frac{J^\prime (\Omega) [\chi]}{\| \varepsilon (\chi) \|_{L^{p^\ast} (\Omega)}} \hspace{0.5cm}
    ( = - \| J^\prime (\Omega) \|_{p^\ast;\perp,{\rm RM}}^\ast )
    \label{eq:steepest_descent_symmetric}
  \end{equation}
  holds.
  \label{thrm-dual_norm_as_best_approximation_symmetric}
\end{theorem}

\begin{proof}
  For each $\chi \in \Theta^{p^\ast}_{\perp,{\rm RM}}$ and each $T \in \Sigma^{p,0}_{f,j,{\rm sym}}$, (\ref{eq:relation_shape_derivative_tensor_bound}) leads to
  \begin{equation}
    \frac{J^\prime (\Omega) [\chi]}{\| \varepsilon (\chi) \|_{L^{p^\ast} (\Omega)}} \leq \| T - K ( u_\Omega , y_\Omega ) \|_{L^p (\Omega)}
    \label{eq:dual_inequality}
  \end{equation}
  which implies
  \begin{equation}
    \| J^\prime (\Omega) \|_{p^\ast;\perp,{\rm RM}}^\ast \leq \inf_{T \in \Sigma^{p,0}_{f,j,{\rm sym}}} \| T - K ( u_\Omega , y_\Omega ) \|_{L^p (\Omega)} \: .
  \end{equation}
  Due to Theorem \ref{theorem-unique_S}, there is an $S \in \Sigma^{p,0}_{f,j,{\rm sym}}$ such that the right-hand side in
  (\ref{eq:dual_norm_as_best_approximation_symmetric}) equals $\| S - K ( u_\Omega , y_\Omega ) \|_{L^p (\Omega)}$. Therefore, in order to establish
  (\ref{eq:dual_norm_as_best_approximation_symmetric}) all that is left to show is the existence of $\vartheta \in \Theta^{p^\ast}_{\perp,{\rm RM}}$ with
  \begin{equation}
    \frac{J^\prime (\Omega) [\vartheta]}{\| \varepsilon (\vartheta) \|_{L^{p^\ast} (\Omega)}} = \| S - K ( u_\Omega , y_\Omega ) \|_{L^p (\Omega)} \: .
    \label{eq:dual_equality}
  \end{equation}
  If we set $\vartheta = - \theta$, where $\theta \in \Theta_{\perp,{\rm RM}}^{p^\ast}$ is the Lagrange multiplier in (\ref{eq:KKT}), then
  (\ref{eq:relation_shape_derivative_tensor}) leads to
  \begin{equation}
    J^\prime (\Omega) [\vartheta] = - J^\prime (\Omega) [\theta] = ( S - K ( u_\Omega,y_\Omega ) , \varepsilon (\theta) ) \: .
  \end{equation}
  Inserting (\ref{eq:multiplier_symmetric_gradient}), we obtain
  \begin{equation}
    \begin{split}
      J^\prime (\Omega) [\vartheta] & = ( S - K ( u_\Omega,y_\Omega ) , | S - K (u_\Omega,y_\Omega) |^{p-2} ( S - K (u_\Omega,y_\Omega) ) ) \\
      & = \| S - K (u_\Omega,y_\Omega) \|_{L^p (\Omega)}^p \: .
    \end{split}
    \label{eq:numerator}
  \end{equation}
  On the other hand, putting norms directly around (\ref{eq:multiplier_symmetric_gradient}), we get
  \begin{equation}
    \begin{split}
      \| \varepsilon (\theta) \|_{L^{p^\ast} (\Omega)} & = \| | S - K (u_\Omega,y_\Omega) |^{p-2} (S - K (u_\Omega,y_\Omega)) \|_{L^{p^\ast} (\Omega)} \\
      & = \left( \int_\Omega | S - K (u_\Omega,y_\Omega) |^{(p-2) p^\ast + p^\ast} dx \right)^{1/p^\ast} \\
      & = \left( \int_\Omega | S - K (u_\Omega,y_\Omega) |^p dx \right)^{(p-1)/p} = \| S - K (u_\Omega,y_\Omega) \|_{L^p (\Omega)}^{p-1} \: .
    \end{split}
    \label{eq:denominator}
  \end{equation}
  Combining (\ref{eq:numerator}) and (\ref{eq:denominator}) proves (\ref{eq:dual_equality}).
  
  Finally, (\ref{eq:steepest_descent_symmetric}) follows from
  \begin{equation}
    \begin{split}
      - \| J^\prime (\Omega) & \|_{p^\ast;\perp,{\rm RM}}^\ast
      = \inf_{\chi \in \Theta_{\perp,{\rm RM}}} \frac{J^\prime (\Omega) [\chi]}{\| \varepsilon (\chi) \|_{L^{p^\ast} (\Omega)}}
      \leq \frac{J^\prime (\Omega) [-\theta]}{\| \varepsilon (-\theta) \|_{L^{p^\ast} (\Omega)}} \\
      & = - \frac{J^\prime (\Omega) [\theta]}{\| \varepsilon (\theta) \|_{L^{p^\ast} (\Omega)}} = - \| S - K ( u_\Omega , y_\Omega ) \|_{L^p (\Omega)}
      \leq - \| J^\prime (\Omega) \|_{p^\ast;\perp,{\rm RM}}^\ast \: ,
    \end{split}
  \end{equation}
  the last inequality following from (\ref{eq:dual_inequality}).
\end{proof}

\begin{remark}
  The second statement in Theorem \ref{thrm-dual_norm_as_best_approximation_symmetric} means that $\theta$ is the steepest descent direction for the
  shape functional in the $W^{1,p^\ast} (\Omega)$ topology endowed with the $L^{p^\ast} (\Omega)$ norm of the symmetric gradient. In the limit, for
  $p^\ast \rightarrow \infty$, this represents another way of accessing the shape gradient with respect to the $W^{1,\infty}$ topology
  recently studied by Deckelnick, Herbert and Hinze in \cite{DecHerHin:22}; see also \cite[Theorem 3.3]{Sta:24}.
\end{remark}

\begin{remark}
  Looking at the optimality conditions (\ref{eq:KKT}), it is obvious that one may insert functions which are not weakly differentiable for $\theta$ but only
  satisfy $\theta \in L^{p^\ast} (\Omega ; \R^d)$. This raises the question if it is also possible to get away with test functions $\chi \in L^{p^\ast} (\Omega ; \R^d)$
  which are discontinuous. To this end, we rewrite the second line in (\ref{eq:KKT}) which, together with the boundary condition on $S$, reads
  \begin{equation}
    ( \div \: S , \chi ) - ( f \: \nabla y_\Omega , \chi ) - ( j (u_\Omega) , \div \: \chi ) - \langle S \cdot n , \left. \chi \right|_{\partial \Omega} \rangle = 0
    \label{eq:constraint_two_combined}
  \end{equation}
  for all $\chi \in W^{1,p^\ast} ( \Omega;\R^d )$ in correspondence to the derivation from the shape derivative (\ref{eq:shape_derivative}). Integrating by parts
  (under the assumption that $j (u_\Omega) \in W^{1,p} (\Omega)$) and separating the conditions in the interior and on the boundary, we get
  \begin{equation}
    \begin{split}
      ( \div \: S , \chi ) - ( f \: \nabla y_\Omega , \chi ) + ( \nabla j (u_\Omega) , \chi ) & = 0 \: , \\
      \langle S \cdot n , \left. \chi \right|_{\partial \Omega} \rangle + \langle j (u_\Omega) , \left. \chi \cdot n \right|_{\partial \Omega} \rangle & = 0 \: ,
    \end{split}
    \label{eq:constraint_two_modified}
  \end{equation}
  where we may choose $\chi \in L^{p^\ast} (\Omega;\R^d)$ for the first constraint and treat the second one independently. This means that
  $S \notin \Sigma^{p,0}$ for the shape tensor approximation in (\ref{eq:constraint_two_modified}) in contrast to the solution of (\ref{eq:KKT}). However, the
  difference between the two is divergence-free, i.e., contained in $\Sigma_{\rm sol}^{p,0}$, and therefore lead to the same best approximation.

  The constraints in (\ref{eq:constraint_two_modified}) may be evaluated for more general, possibly discontinuous, test functions $\chi$ and
  $\left. \chi \right|_{\partial \Omega}$, respectively, which will be important for the finite element discretization in the following section.
  In particular, under the additional assumption that $j (0) = 0$, the second boundary term in (\ref{eq:constraint_two_modified}) vanishes and we
  may again seek $S \in \Sigma^{p,0}$, now with the modified constraint from (\ref{eq:constraint_two_modified}). This leads to the following optimality system
  for $S \in \Sigma^{p,0}$ and the Lagrange multipliers $\theta \in \Theta^{p^\ast}_{\perp,{\rm RM}}$ and $\omega \in \Xi^{p^\ast}$:
  \begin{equation}
    \begin{split}
      ( | S - K (u_\Omega,y_\Omega) |^{p-2} (S - K (u_\Omega,y_\Omega)) , T ) \hspace{1.5cm} & \\
      + ( \div \: T , \theta ) + ( \as \: T , \omega ) & = 0 \mbox{ for all } T \in \Sigma^{p,0} \: , \\
      ( \div \: S , \chi ) - ( f \: \nabla y_\Omega , \chi ) + ( \nabla j (u_\Omega) , \chi ) & = 0 \mbox{ for all } \chi \in \Theta^{p^\ast}_{\perp,{\rm RM}} \: , \\
      ( \as \: S , \gamma ) & = 0 \mbox{ for all } \gamma \in \Xi^{p^\ast} \: .
    \end{split}
    \label{eq:KKT_2}
  \end{equation}
\end{remark}

\section{Discretization with weakly symmetric tensor finite elements}
\label{sec-discretization}

On a triangulation $\cT_h$ of a polyhedrally bounded approximation $\Omega_h$ of $\Omega$, we use the lowest-order PEERS finite element combination
(cf. \cite{ArnBreDou:84} for the original two-dimensional element and \cite[Sect. 9.4]{BofBreFor:13} for the generalization to three dimensions) for the
approximate solution of (\ref{eq:KKT_2}). The finite-dimensional subspace $\Sigma_h^{p,0} \subset \Sigma^{p,0}$ consists of lowest-order Raviart-Thomas
functions enriched with the curl of suitable element bubbles for each row. The discrete Lagrange multiplier spaces are $\Theta_h$ consisting of piecewise
constants for the (component-wise) approximation of the deformation field $\theta \in \Theta^{p^\ast}_{\perp,{\rm RM}}$ and continuous piecewise linear functions
for each skew-symmetric matrix entry in $\Xi_h$ approximating $\omega \subset \Xi^{p^\ast}$. The mean value zero constraints in the definition of $\Xi^{p^\ast}$
in (\ref{eq:rotation_space}) are enforced by additional constraints (one or three in the two- or three-dimensional case, respectively) on $\omega_h \in \Xi_h$
with additional Lagrange multipliers.
Note that these constraints also enforce the mean value zero constraints connected to the rotations in the definition of $\Theta^{p^\ast}_{\perp,{\rm RM}}$ in
(\ref{eq:deformation_orthogonal_RM}) and vice versa (see the proof of Theorem 3.3 towards the end). This is fortunate since we could not enforce these
constraints in the subspace $\Theta_h$ of discontinuous functions so easily. The other mean value constraints on $\theta_h \in \Theta_h$ associated with
the (two or three, respectively) constant translations can again be enforced as additional constraints with real numbers as Lagrange multipliers.

We end up with the following finite element version of the optimality system (\ref{eq:KKT_2}): Find $S_h \in \Sigma_h^{p,0}$, $\theta_h \in \Theta_h$,
$\omega_h \in \Xi_h$ and $\lambda \in \R^d$, $\mu \in \R^{d(d-1)/2}$ such that
\begin{equation}
  \begin{split}
    ( | S_h - K (u_{\Omega,h},y_{\Omega,h}) |^{p-2} (S_h - K (u_{\Omega,h},y_{\Omega,h})) , T_h ) \hspace{2.5cm} \\
    + ( \div \: T_h , \theta_h ) + ( \as \: T_h , \omega_h ) & = 0 \: , \\
    ( \div \: S_h , \chi_h ) - ( f \: \nabla y_{\Omega,h} , \chi_h ) + ( \nabla j (u_{\Omega,h}) , \chi_h ) & = 0 \: , \\
    ( \as \: S_h , \gamma_h ) & = 0 \: , \\
    ( \theta_h , \zeta ) & = 0 \: , \\
    ( \omega_h , A^d (\xi) ) & = 0
  \end{split}
  \label{eq:KKT_h}
\end{equation}
holds for all $T_h \in \Sigma_h^{p,0}$, $\chi_h \in \Theta_h$, $\gamma_h \in \Xi_h$, $\zeta \in \R^d$ and $\xi \in \R^{d(d-1)/2}$. In the last line of (\ref{eq:KKT_h}),
the notation
\begin{equation}
  A^2 (\xi) = \begin{pmatrix} \;\;0 & \xi \\ -\xi & 0 \end{pmatrix} \: , \:
  A^3 (\xi) = \begin{pmatrix} \;\;0 & \;\;\xi_3 & -\xi_2 \\ -\xi_3 & \;\;0 & \;\;\xi_1 \\ \;\;\xi_2 & -\xi_1 & \;\;0 \end{pmatrix}
\end{equation}
is used. Here, $u_{\Omega,h}$ and $y_{\Omega,h}$ are conforming piecewise linear finite element approximations of (\ref{eq:bvp}) and (\ref{eq:adjoint_bvp}),
respectively. $\eta_{p,h} (\Omega_h) := \| S_h - K (u_{\Omega,h},y_{\Omega,h}) \|_{L^p (\Omega_h)}$ are approximate values for the distance of the shape tensor
$K (u_\Omega,y_\Omega)$ to $\Sigma^{p,0}_{f,j,{\rm sym}}$ and therefore also for the dual norm of the shape derivative
$\| J^\prime (\Omega) \|_{p^\ast;\perp,{\rm RM}}^\ast$.

The following rather simple fixed point scheme is used to handle the nonlinearity in the first line of (\ref{eq:KKT_h}): Determine, in step $k$,
$S_h^{(k)} \in \Sigma_h^{p,0}$, $\theta_h \in \Theta_h$, $\omega_h \in \Xi_h$ and $\lambda \in \R^d$, $\mu \in \R^{d(d-1)/2}$ such that
\begin{equation}
  \begin{split}
    ( | S_h^{(k-1)} - K (u_{\Omega,h},y_{\Omega,h}) |^{p-2} (S_h^{(k)} - K (u_{\Omega,h},y_{\Omega,h})) , T_h ) \hspace{2cm} \\
    + ( \div \: T_h , \theta_h ) + ( \as \: T_h , \omega_h ) & = 0 \: , \\
    ( \div \: S_h^{(k)} , \chi_h ) - ( f \: \nabla y_{\Omega,h} , \chi_h ) + ( \nabla j (u_{\Omega,h}) , \chi_h ) & = 0 \: , \\
    ( \as \: S_h^{(k)} , \gamma_h ) & = 0 \: , \\
    ( \theta_h , \zeta ) & = 0 \: , \\
    ( \omega_h , A^d (\xi) ) & = 0
  \end{split}
  \label{eq:KKT_iterative}
\end{equation}
is satisfied for all $T_h \in \Sigma_h^{p,0}$, $\chi_h \in \Theta_h$, $\gamma_h \in \Xi_h$, $\zeta \in \R^d$ and $\xi \in \R^{d(d-1)/2}$. Our experience is that this
works well as long as $p$ is not too close to 1. The value of $p = 1.1$ that we used in the computational experiments presented below are well on the safe
side where (\ref{eq:KKT_iterative}) converges and also the reconstruction process detailed in the next section provides accurate results.

\noindent
{\bf Example 1 (Circular optimal shape).}
\label{example-circle}
We illustrate the behavior of the computed values for $\eta_{p,h} (\Omega_h)$ for a very simple test problem where the optimal shape is a disk.
As right-hand side for (\ref{eq:bvp}) we choose $f = -1/2$ inside the unit disk $D \subset \R^2$ and $f = 1/2$ outside, while we set $j (u_\Omega) \equiv u_\Omega/2$
for the definition of the shape functional. The shape tensor for disks $D_R$ of radius $R > 1$ can be explicitly calculated from the solutions of (\ref{eq:bvp}) and
(\ref{eq:adjoint_bvp}),
\begin{equation}
  u_R (x) = \left\{
  \begin{array}{ll}
    \frac{R^2 + | x |^2}{8} - \frac{1}{4} - \frac{1}{2} \ln R \: , \: & 0 \leq | x | < 1 \: , \\
    \frac{R^2 - | x |^2}{8} + \frac{1}{2} \ln \frac{|x|}{R} \: , \: & 1 < | x | < R \: ,
  \end{array} \right. \hspace{1cm}
  y_R (x) = \frac{| x |^2 - R^2}{8} \: ,
  \label{eq:bvp_solutions_example_1}
\end{equation}
as
\begin{equation}
  \begin{split}
    K ( u_R , y_R ) & = \left( \nabla u_R \cdot \nabla y_R \right) I - \nabla y_R \otimes \nabla u_R - \nabla u_R \otimes \nabla y_R \\
    & = \left\{ \begin{array}{ll}
      \frac{1}{16} | \id |^2 I - \frac{1}{8} \id \otimes \id & \: , \: 0 \leq | x | < 1 \: , \\
      \left( \frac{1}{8} - \frac{1}{16} | \id |^2 \right) I - \left( \frac{1}{4 \: | \id |^2} - \frac{1}{8} \right) \id \otimes \id & \: , \: 1 < | x | < R
    \end{array} \right.
  \end{split}
  \label{eq:tensor_example_1}
\end{equation}
(see \cite{Sta:24}). It can be seen that, for $R = \sqrt{2}$, $S = K ( u_{\sqrt{2}} , y_{\sqrt{2}} )$ itself satisfies (\ref{eq:KKT_2}) which means that this is an optimal shape.

We compute approximations $\eta_{1.1,h} (\Omega)$ to $\| S - K (u_\Omega,y_\Omega) \|_{L^p (\Omega)}$ for different domains $\Omega$ of variable closeness
to the optimal shape shown in Figure \ref{fig-shapeex1}, which are all scaled to have the same area $| \Omega | = 2 \pi$ (shared by the optimal shape $D_{\sqrt{2}}$):
a square (solid line), an octagon (dashed line) and a hexadecagon (16 edges, dotted line).

\begin{figure}[h!]
  \hspace{1cm}\includegraphics[scale=0.32]{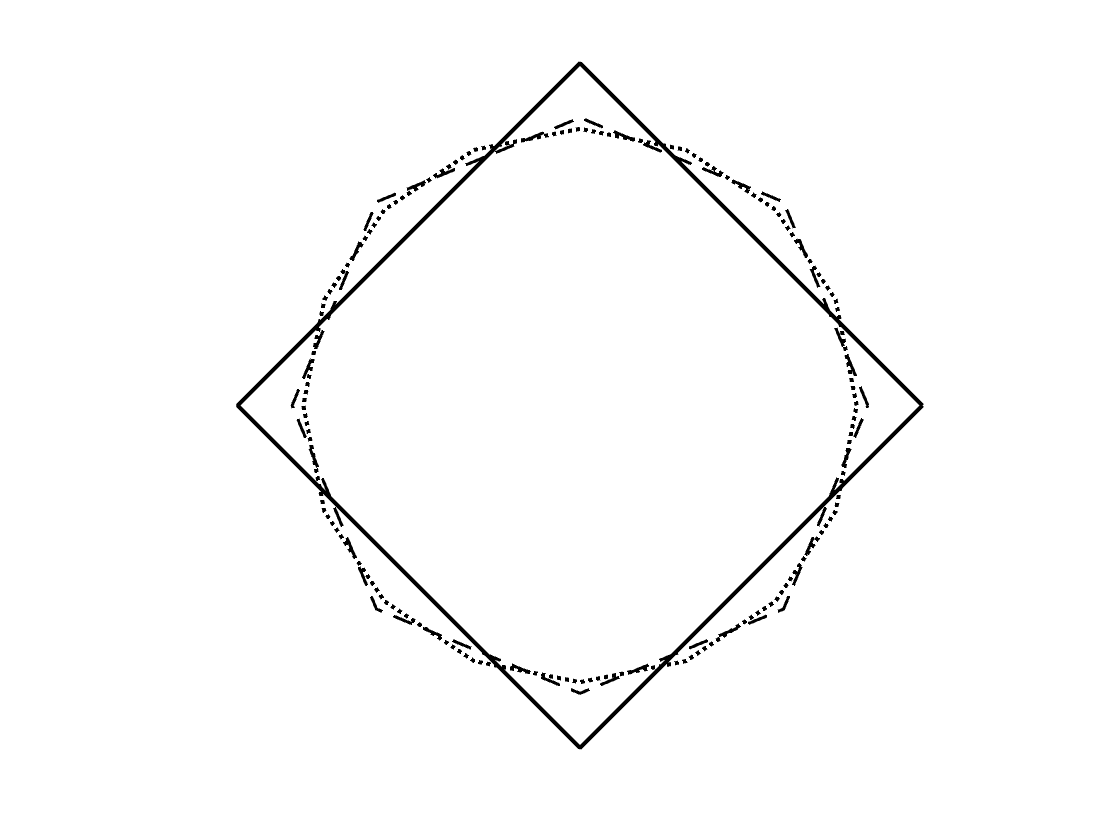}
  \caption{Example 1: Polygonal domains approximating a disk}
  \label{fig-shapeex1}
\end{figure}

\begin{figure}[h!]
  \hspace{1cm} \includegraphics[scale=0.26]{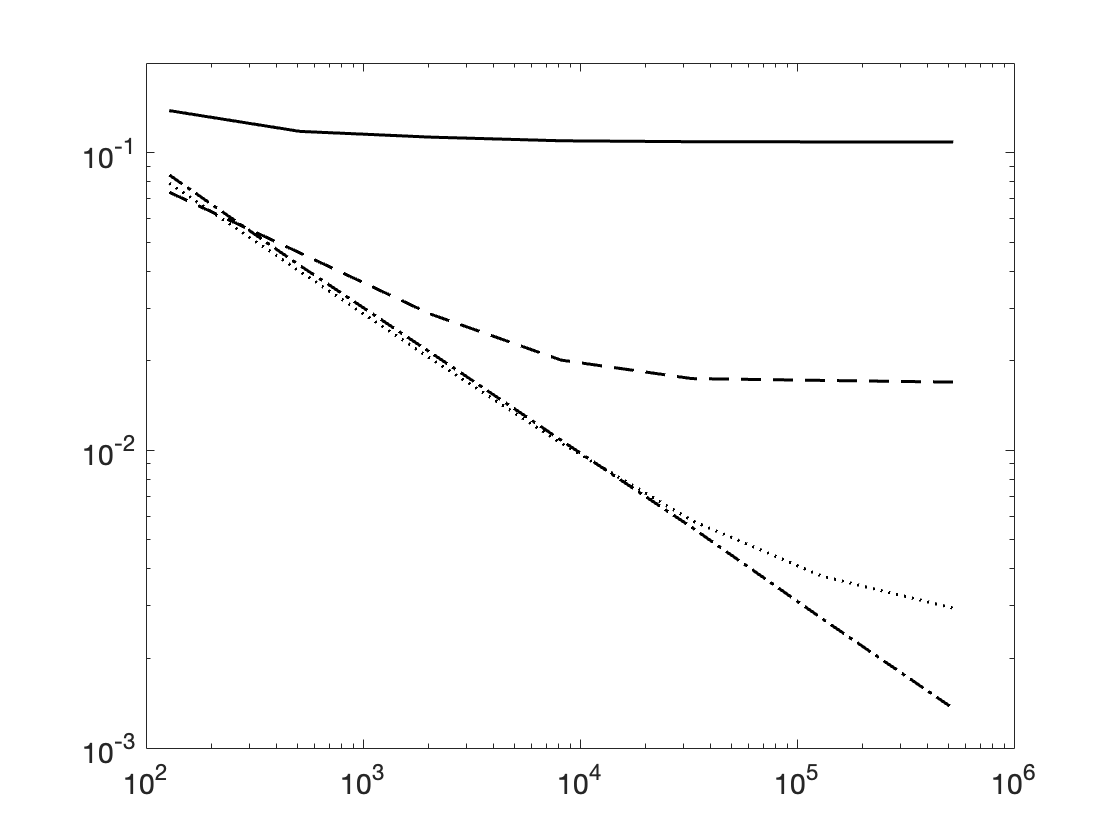}
  \caption{Example 1: Values $\eta_{1.1,h} (\Omega_h)$ for square, octagon and hexadecagon}
  \label{fig-shapeex1delta1point1}
\end{figure}

Figure \ref{fig-shapeex1delta1point1} shows the behavior of $\eta_{1.1,h} (\Omega)$ in dependence of the number of elements for $p = 1.1$ on a sequence of
triangulations resulting from uniform refinement. For the square (solid line), the functional stagnates almost from the beginning at its distance to stationarity.
For the octagon (dashed line), it is reduced somewhat until it stagnates at a lower value, meaning that it is closer, in the mathematical strict sense of Theorem
\ref{thrm-dual_norm_as_best_approximation_symmetric}, to being a stationary shape. For the hexadecagon (dotted line), the behavior is repeated with an even
smaller asymptotical value. Finally, the dash-dotted line shows the behavior of $\eta_{1.1,h}$ for successively closer polygonal approximations to the optimal disk
$D_{\sqrt{2}}$ tending to zero at a rate proportionally to $h$ (the inverse of the square root of the number of elements). In this case, $\eta_{p,h} (\Omega_h)$ also
contains the geometry error associated with the approximation $\Omega_h$ to $D_{\sqrt{2}}$. In contrast, we have $\Omega_h = \Omega$ for the polygonal
domains above and the difference between $\eta_{p,h} (\Omega)$ and $\| S - K (u_\Omega,y_\Omega) \|_{L^p (\Omega)}$ is solely caused by the approximation
error associated with solving (\ref{eq:KKT_h}) and the underlying boundary value problems (\ref{eq:bvp}) and (\ref{eq:adjoint_bvp}). The results shown in
Figure \ref{fig-shapeex1delta1point1} look quite similar to those obtained in \cite{Sta:24} without the weak symmetry condition although the difference between
the different polygonal shapes seems to initiate at somewhat coarser meshes. This alone, however, would not be a convincing argument for the use of the
more complicated approach presented in this paper. The true advantage of working with the weak symmetry condition lies in the behavior of the shape gradient
iteration presented in the following section.

\section{Shape gradient iteration with reconstructed deformations in $W^{1,p^\ast} (\Omega)$}
\label{sec-shape_gradient_iteration}

The finite element approximation $\theta_h$ obtained from (\ref{eq:KKT_h}) for the steepest descent deformations is a discontinuous function, in general, and
therefore not suitable for computing a deformed shape $(\id + \theta_h) (\Omega_h)$. Instead, a continuous deformation $\theta_h^\diamond$ is
required in order to lead to an admissible domain $(\id + \theta_h^\diamond) (\Omega_h)$ with a non-degenerate triangulation $(\id + \theta_h^\diamond) (\cT_h)$.
While $\theta_h$ is already a good approximation to $\theta$ in the $L^{p^\ast} (\Omega;\R^d)$-norm, we also know from the first equation in (\ref{eq:KKT_h}) that
$| S_h - K (u_{\Omega,h},y_{\Omega,h}) |^{p-2} (S_h - K (u_{\Omega,h},y_{\Omega,h})) + \omega_h$ is a good approximation to
\begin{equation}
  | S - K (u_\Omega,y_\Omega) |^{p-2} (S - K (u_\Omega,y_\Omega)) + \omega = \nabla \theta \: .
\end{equation}
This knowledge can be exploited in the following local potential reconstruction procedure for a continuous deformation
$\theta_h^\diamond \in W^{1,p^\ast} (\Omega_h;\R^d)$:\\
{\bf Algorithm 1.}\\
1. For each element $\tau \in \cT_h$, compute $\left. \nabla \theta_h^\Box \right|_\tau \in \R^{d \times d}$ such that
\[
  \| \nabla \theta_h^\Box - \left| S_h - K ( u_{\Omega,h} , y_{\Omega,h} ) \right|^{p-2} ( S_h - K ( u_{\Omega,h} , y_{\Omega,h} ) - \omega_h \|_{L^{p^\ast} (\tau)}
  \longrightarrow \min!
\]
2. For each element $\tau \in \cT_h$, compute $\theta_h^\Box$ with $\left. \nabla \theta_h^\Box \right|_\tau$ given by step 1 such that
\[
  \| \theta_h^\Box - \theta_h \|_{L^{p^\ast} (\tau)} \longrightarrow \min!
\]
3. Compute $\theta_h^\diamond$ continuous and piecewise linear by averaging at the vertices:
\[
  \theta_h^\diamond (\nu) = \frac{1}{| \{ \tau \in \cT_h : \nu \in \tau \} |} \sum_{\tau \in \cT_h : \nu \in \tau} \theta_h^\Box (\left. \nu \right|_\tau) \: .
\]
Obviously, the above algorithm involves only local computations and the only difference with respect to the one in \cite{Sta:24} is that $\omega_h$ is added in
the first step. The main computational work consists in solving nonlinear systems of equations of dimension $d^2$ and $d$, respectively, in steps 1 and 2 for
each element $\tau \in \cT_h$. The algorithm coincides with (the lowest-order case of) the classical reconstruction technique by Stenberg (see \cite{Ste:91},
where it is presented for $p^\ast = 2$). Alternatively, a related reconstruction procedure based on a vertex patch decomposition (see \cite[Sect. 3]{ErnVoh:15})
could be used.

Now $\theta_h^\diamond$ can be used to compute a deformed shape $\Omega_h^\diamond = (\id + \alpha \theta_h^\diamond) (\Omega_h)$ with a suitable
step-size $\alpha \in \R$. In our case, $\alpha$ is determined as the solution of the following discrete minimization problem based on the shape functional:
\begin{equation}
  \alpha = \beta^{k^\ast} \: , \: k^\ast = \arg \min_{k \in \Z} J ((\id + \beta^k \theta_h^\diamond) (\Omega_h))
  \label{eq:stepsize_discrete}
\end{equation}
(with $\beta = 1.25$ in our computational experiments).

\noindent
{\bf Example 1; continued.}
The above shape gradient iteration is tested by starting with a tetragon as initial guess and observing how close to the optimal shape $D_{\sqrt{2}}$
one gets. Figure \ref{fig-example_1_circle_from_tetragon} shows the boundary of the final shape $\Omega_h^\diamond$ for $p = 2$ (on the left) and
$p = 1.1$ (on the right) for a triangulation after 7 uniform refinements (131072 elements). Apparently, the right angles from the tetragon have been
straightened out more by the shape gradient iteration for the smaller value of $p$. Also, both shapes look closer to the optimal circle than the corresponding
results without the weak symmetry constraint shown in \cite{Sta:24}.

\begin{figure}[h!]

  \vspace{-0.2cm}

  \hspace{-1.5cm}\includegraphics[scale=0.195]{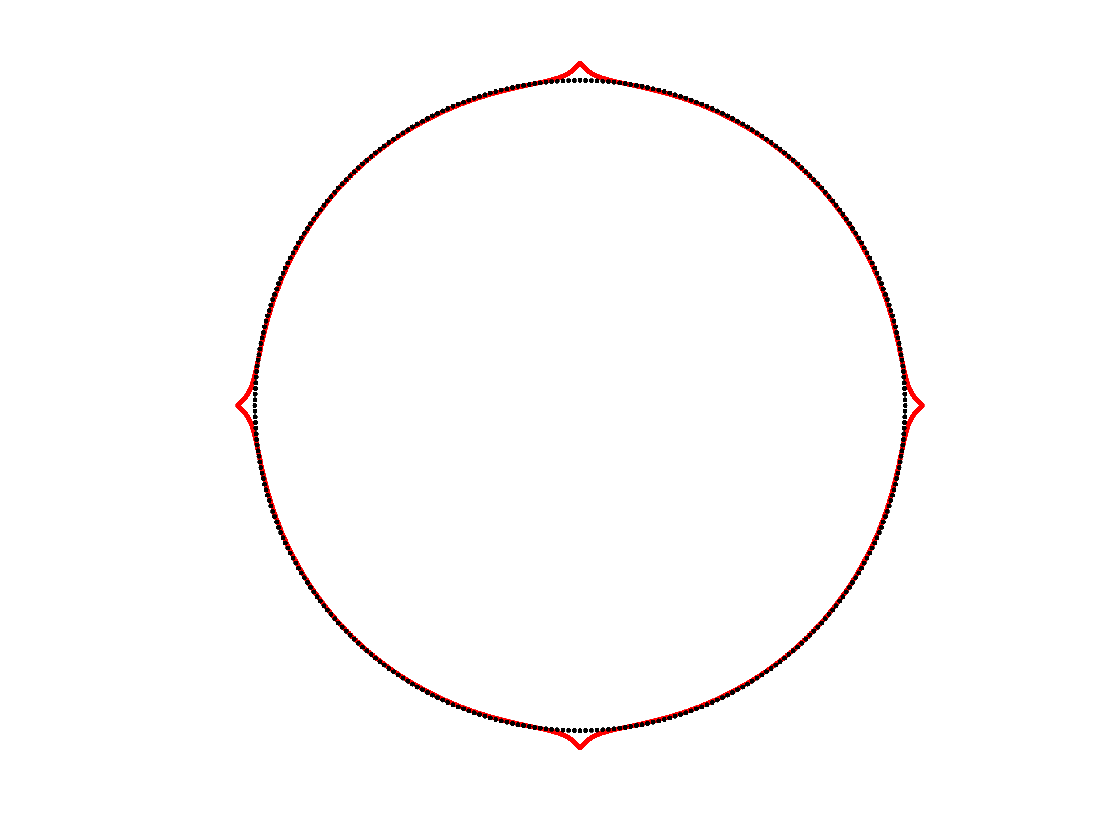}\hspace{-1.35cm}\includegraphics[scale=0.195]{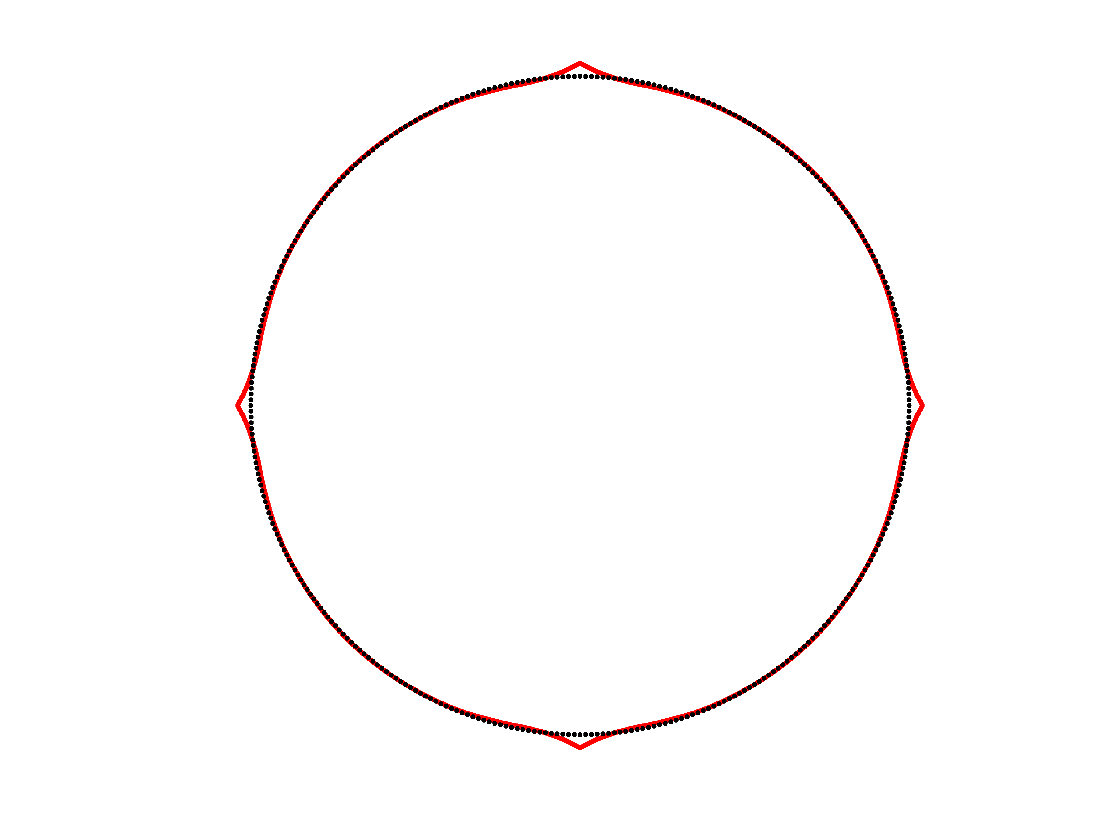}
  
  \vspace{-0.2cm}
  
  \caption{Result of shape gradient iteration for $p = 2$ (left) and $p = 1.1$ (right)}
  \label{fig-example_1_circle_from_tetragon}
\end{figure}



\noindent
{\bf Example 2 (``gingerbread man'').}
This example constitutes a more challenging shape optimization problem to test our approach. The right-hand side in the variational problem (\ref{eq:bvp})
is changed to
\begin{equation}
  \begin{split}
    f (x) & = - \frac{1}{2} + \frac{4}{5} \: | x |^2 + 2 \sum_{i=1}^5 \exp \left( - 8 \: | x - y^{(i)} |^2 \right)
    - \sum_{i=1}^5 \exp \left( - 8 \: | x - z^{(i)} |^2 \right) \\
    \mbox{ with } & \;\;\;\; y^{(i)} = \left( \sin ( \frac{(2 i + 1) \pi}{5} ) , \cos ( \frac{(2 i + 1) \pi}{5} ) \right) \: , \: i = 1 , \ldots , 5 \: , \\
    & \;\;\;\; z^{(i)} = \left( \frac{6}{5} \sin ( \frac{2 i \pi}{5} ) , \frac{6}{5} \cos ( \frac{2 i \pi}{5} ) \right) \: , \: i = 1 , \ldots , 5
  \end{split}
\end{equation}
(see \cite{BarWac:20}) while everything else remains unchanged with respect to Example 1. If we initiate the shape gradient iteration with the unit disk, we obtain
the final shapes shown in Figure \ref{fig-gingerbread_levels4and5} (on the left for $p = 2$ and on the right for $p = 1.1$, on the top for refinement level 4 with 2048
elements and on the bottom for level 5 with 8192 elements).

\begin{figure}[h!]

  \vspace{-0.2cm}
  
  \hspace{-1.65cm}\includegraphics[scale=0.2]{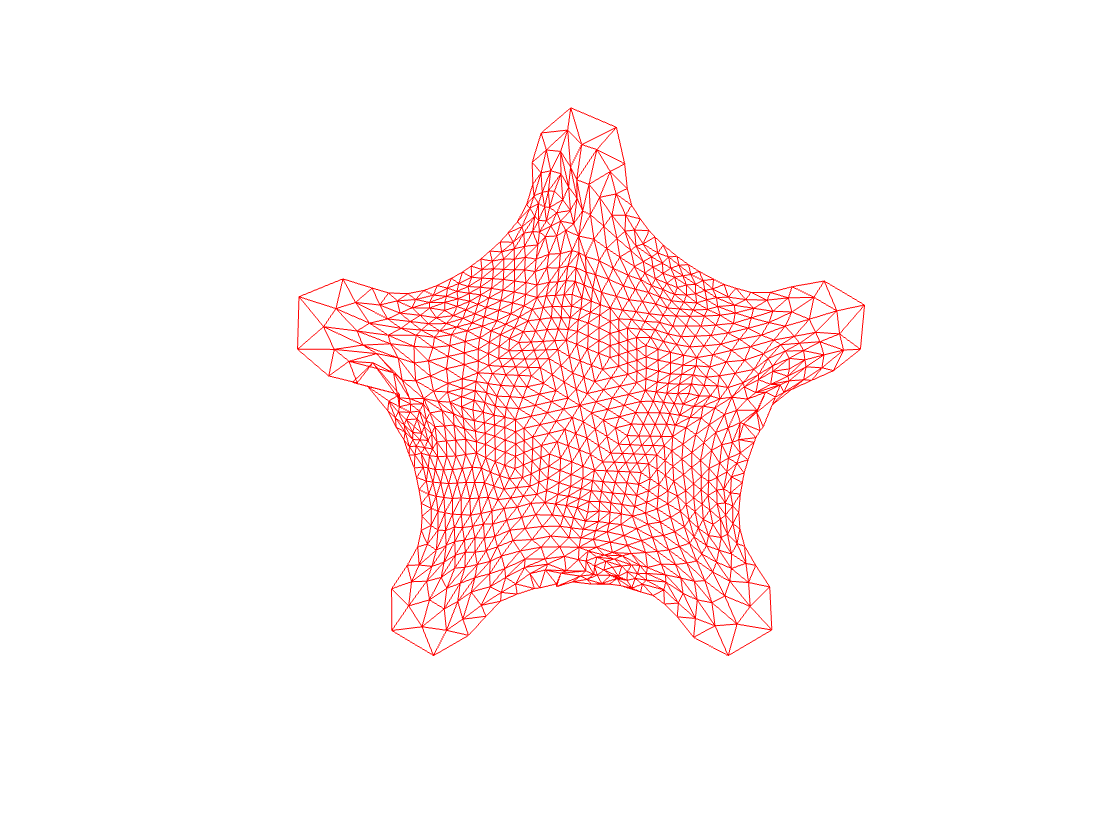}\hspace{-1.5cm}\includegraphics[scale=0.2]{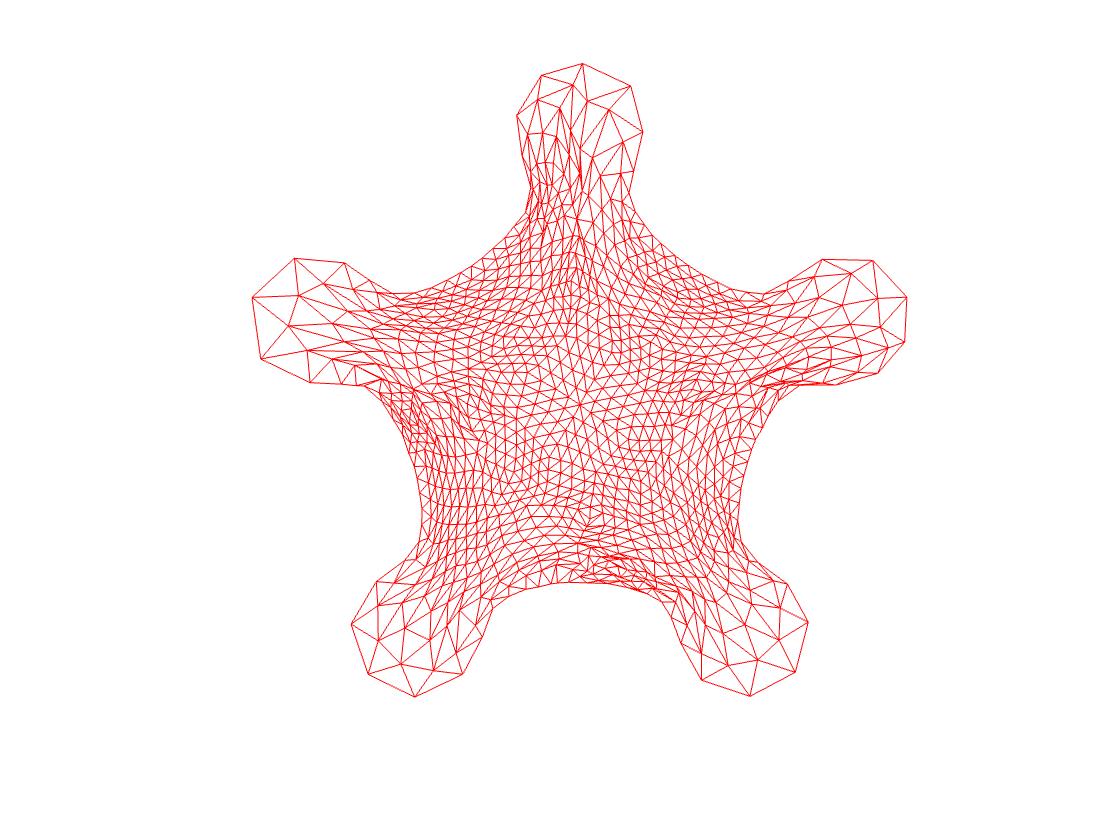}

  \vspace{-0.5cm}

  \hspace{-1.65cm}\includegraphics[scale=0.2]{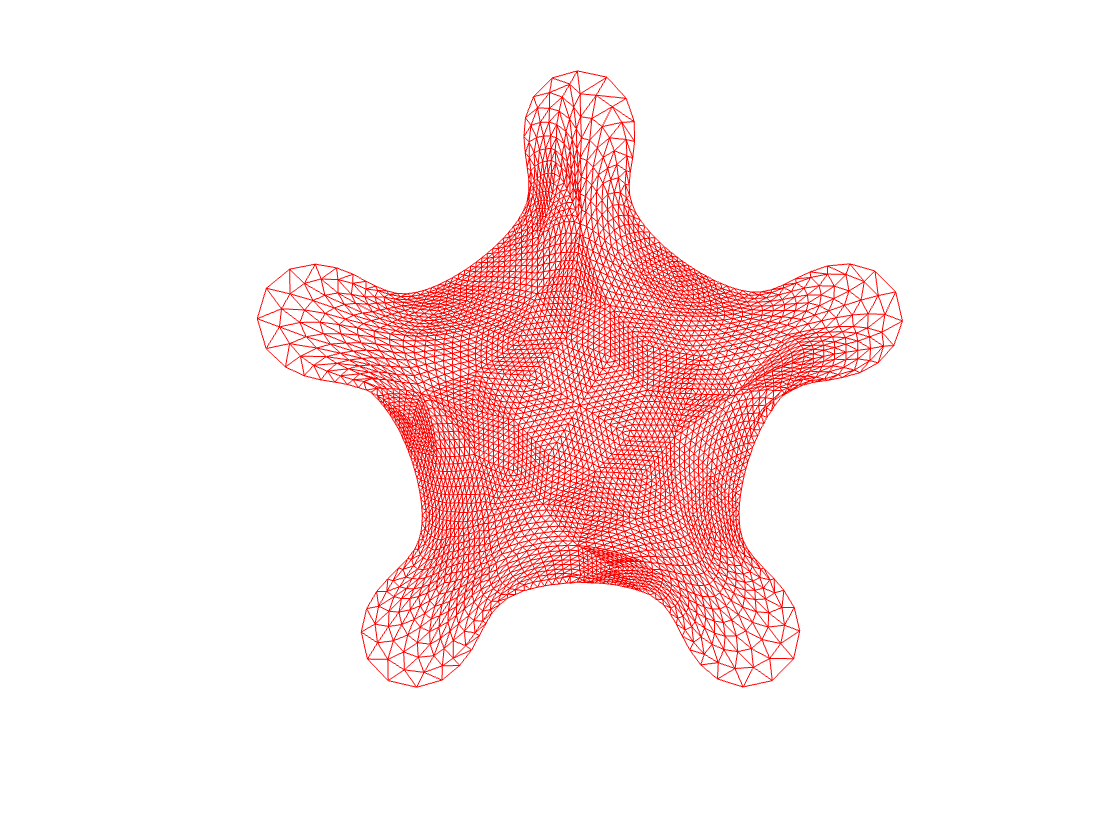}\hspace{-1.5cm}\includegraphics[scale=0.2]{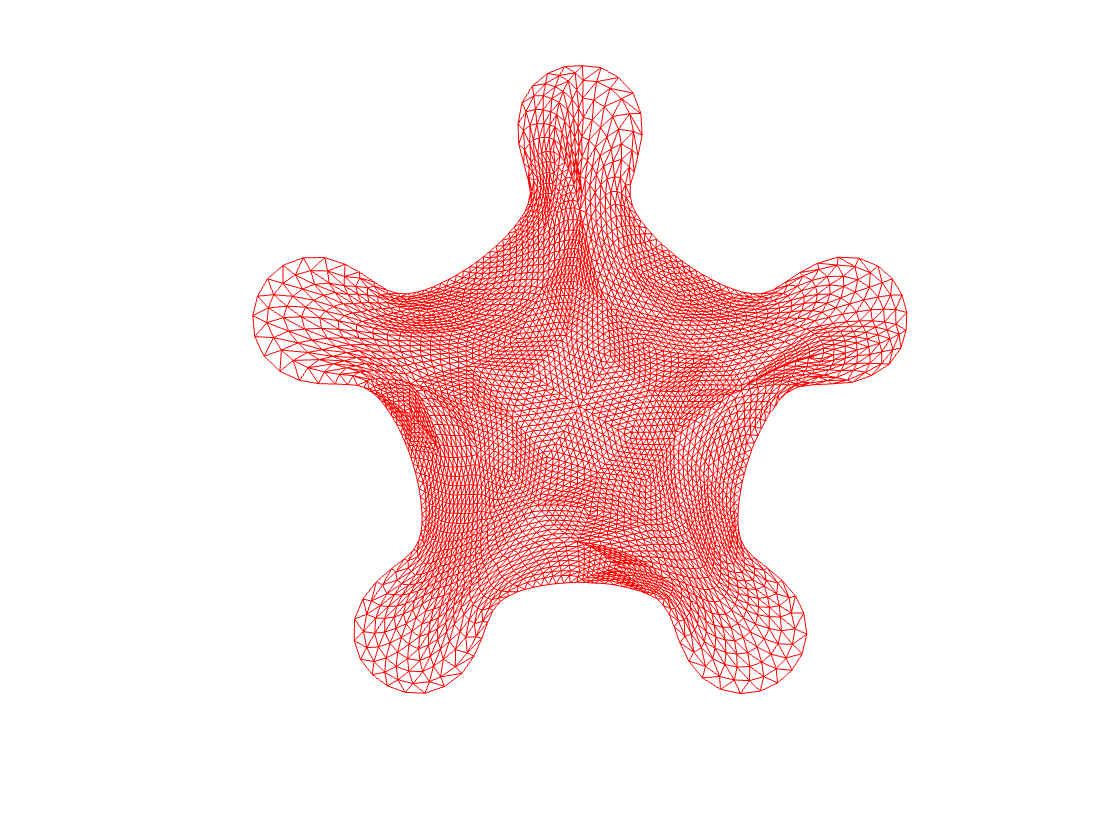}
  
  \vspace{-0.5cm}
  
  \caption{Final shapes for $p = 2$ (left) and $p = 1.1$ (right) on levels 4 (top) and 5 (bottom)}
  \label{fig-gingerbread_levels4and5}
\end{figure}

The two triangulations shown in the upper half of Figure \ref{fig-gingerbread_levels4and5}, associated with refinement level 4, are both severely distorted. For $p = 2$,
shown on the left, the degeneration causes the shape gradient iteration to terminate prematurely resulting in a shape approximation that is visibly different from
the optimal one. The result for $p = 1.1$ shown on the right, however, looks significantly better. For the resulting shapes at refinement level 5, shown in the lower
half of the figure, the improvement for $p = 1.1$ compared to $p = 2$ is much less pronounced. This behavior is also illustrated by the numbers in Table
\ref{table-example_2_shapiter_2} and Table \ref{table-example_2_shapiter_1point1}, where the functional value reached for $p = 1.1$ is lower, on all refinement
levels, than the one for $p = 2$. In particular, the values obtained for $p = 2$ on refinement level 4 and to some extent on level 5 are still quite far away from
the minimum.

\begin{table}[h!]
  \centering
  {\small
  \begin{tabular}{|c|cccc|}
    \hline $| \cT_h |$ & 2048 & 8192 & 32768 & 131072 \\ \hline
    $J (\Omega_h^\diamond)$ & $-1.447556 \cdot 10^{-2}$ & $-1.492032 \cdot 10^{-2}$ & $-1.496967 \cdot 10^{-2}$ & $-1.496967 \cdot 10^{-2}$ \\
    $\eta_{2,h} (\Omega_h^\diamond)$ & \;\;$2.8112 \cdot 10^{-3}$ & \;\;$1.5110 \cdot 10^{-3}$ & \;\;$4.7128 \cdot 10^{-4}$ & \;\;$2.4289 \cdot 10^{-4}$ \\ \hline
  \end{tabular}
  }
  \caption{Shape approximation for $p = 2$}
  \label{table-example_2_shapiter_2}
\end{table}

\begin{table}[h!]
  \centering
  {\small
  \begin{tabular}{|c|cccc|}
    \hline $| \cT_h |$ & 2048 & 8192 & 32768 & 131072 \\ \hline
    $J (\Omega_h^\diamond)$ & $-1.496129 \cdot 10^{-2}$ & $-1.496941 \cdot 10^{-2}$ & $-1.496972 \cdot 10^{-2}$ & $-1.496980 \cdot 10^{-2}$ \\
    $\eta_{1.1,h} (\Omega_h^\diamond)$ \!\!\! & \;\;$2.4922 \cdot 10^{-3}$ &\; \;$1.2436 \cdot 10^{-3}$ & \;\;$6.2465 \cdot 10^{-4}$ & \;\;$3.1597 \cdot 10^{-4}$ \\ \hline
  \end{tabular}
  }
  \caption{Shape approximation for $p = 1.1$}
  \label{table-example_2_shapiter_1point1}
\end{table}

\section{Adjustment of the rigid body modes}
\label{sec-rigid_body_modes}

In our examples above, the barycenter of the optimal shape was known to be the origin for symmetry reasons. Also the orientation of the optimal shape was
correctly approximated since we started with a disk as initial domain. In general, however, the barycenter may be unknown and translations need to be added
to our deformations. Moreover, adding rotations to the deformations in our shape gradient iteration may also be beneficial for its convergence if the initial
domain has a different orientation than the optimum one. We propose the following two-step procedure, where the rigid body modes are selected in the
first half-step and the second half-step consists of the procedure outlined in the previous section:
\begin{equation}
  \begin{split}
    \Omega_h^\odot & = \Omega_h + \rho^\odot \: , \\
    \Omega_h^\diamond & = (\id + \alpha \theta_h^\diamond) (\Omega_h^\odot) \: .
  \end{split}
\end{equation}
The steepest descent direction $\rho^\odot$ can be computed from
\begin{equation}
  \frac{J^\prime (\Omega) [\rho^\odot]}{||| \rho^\odot |||} = \inf_{\rho \in {\rm RM}} \frac{J^\prime (\Omega) [\rho]}{\| \rho \|_{W^{1,p} (\Omega)}} \: .
  \label{eq:gradient_RM}
\end{equation}
This is a rather simple low-dimensional optimization problem, particularly, since the shape derivative (\ref{eq:shape_derivative_pre}) simplifies to
$J^\prime (\Omega) [\rho] = ( f \: \nabla y_\Omega , \rho )$ for all $\rho \in {\rm RM}$.


\section{Extension to Problems with Constraints on the Shape}
\label{sec-constraints}

Often, shape optimization problems like those given by (\ref{eq:shape_functional}) and (\ref{eq:bvp}) come with an additional real-valued constraint $C (\Omega) = 0$
on the admissible shapes. For example, the area or perimeter for two-dimensional shapes or the volume or surface area for three-dimensional ones might be prescribed.
Given some $\Omega \subset \R^d$ satisfying $C (\Omega) = 0$, we restrict our deformations to the tangent space with respect to this constraint, i.e., to those
$\theta \in \Theta_{\perp,{\rm RM}}^{p^\ast}$ satisfying $C^\prime (\Omega) [\theta] = 0$. We denote this subspace by
\begin{equation}
  \Theta_{\perp,{\rm RM},\parallel}^{p^\ast} = \{ \chi \in \Theta_{\perp,{\rm RM}}^{p^\ast} : C^\prime (\Omega) [\chi] = 0 \} \: .
  \label{eq:tangent_space}
\end{equation}
Again, we would like to gain access to the norm of the corresponding tangential shape derivative
\begin{equation}
  \| J^\prime (\Omega) \|_{p^\ast;\perp,{\rm RM},\parallel}^\ast
  = \sup_{\chi \in \Theta_{\perp,{\rm RM},\parallel}} \frac{J^\prime (\Omega) [\chi]}{\| \varepsilon (\chi) \|_{L^{p^\ast} (\Omega)}} \: .
  \label{eq:dual_norm_elastic_tangent_space}
\end{equation}
Assuming that $C^\prime (\Omega)$ is continuous on $W^{1,p^\ast} (\Omega;\R^d)$ and that $C^\prime (\Omega) [\rho] = 0$ for all rigid body
deformations $\rho \in {\rm RM} (\R^d)$, we have that $C^\prime (\Omega)$ is already determined by its action on $\varepsilon (\chi)$. We may therefore write
\begin{equation}
  C^\prime (\Omega) [\chi] = ( Z^\perp , \varepsilon (\chi) ) \mbox{ with some } Z^\perp \in L^p (\Omega;\R^{d \times d}) \: .
  \label{eq:constraint_derivative_representation}
\end{equation}
Obviously, $Z^\perp$ may be chosen symmetric, i.e., in such a way that its anti-symmetric part vanishes: $\as \: Z^\perp = 0$. The subspace
\begin{equation}
  \Upsilon^{p,\perp}_{{\rm RM},\parallel}
  = \{ Z \in L^p (\Omega;\R^{d \times d}) : ( Z , \varepsilon (\theta) ) = 0 \mbox{ for all } \theta \in \Theta_{\perp,{\rm RM},\parallel}^{p^\ast} \}
  \label{eq:annihilator}
\end{equation}
is called the annihilator of $\Theta_{\perp,{\rm RM},\parallel}^{p^\ast}$ (cf. \cite[Sect. I.2.6]{Kat:80}). From (\ref{eq:tangent_space}) and
(\ref{eq:constraint_derivative_representation}) we deduce that $Z^\perp \in \Upsilon^{p,\perp}$.
Going through the proof of Theorem \ref{thrm-dual_norm_as_best_approximation_symmetric}, we see that the result relies on the existence of $\theta$ such
that
\begin{equation}
  ( S - K (u_\Omega,y_\Omega) , \varepsilon (\theta) ) = \| S - K (u_\Omega,y_\Omega) \|_{L^p (\Omega)} \| \varepsilon (\theta) \|_{L^{p^\ast} (\Omega)} \: ,
  \label{eq:equality_in_Hoelder}
\end{equation}
i.e., equality in H\"older's inequality in (\ref{eq:relation_shape_derivative_tensor_bound}), holds. This relation is satisfied if and only if $\varepsilon (\theta)$
is a multiple of $| S - K (u_\Omega,y_\Omega) |^{p-2} ( S - K (u_\Omega,y_\Omega) )$. For $\theta \in \Theta_{\perp,{\rm RM},\parallel}^{p^\ast}$ this implies
\begin{equation}
  ( Z^\perp , | S - K (u_\Omega,y_\Omega) |^{p-2} ( S - K (u_\Omega,y_\Omega) ) ) = 0
  \label{eq:tensor_projection}
\end{equation}
is satisfied, an additional condition on the shape tensor approximation. This motivates the following generalization of Theorem
\ref{thrm-dual_norm_as_best_approximation_symmetric}.

\begin{theorem}
  The dual norm of the shape derivative (\ref{eq:dual_norm_elastic_tangent_space}) is given by the following best approximation to the shape tensor
  $K (u_\Omega,y_\Omega)$:
  \begin{equation}
    \| J^\prime (\Omega) \|_{p^\ast;\perp,{\rm RM},\parallel}^\ast
    = \inf_{T \in \Sigma^{p,0}_{f,j,{\rm sym}} , \beta \in \R} \| T + \beta Z^\perp - K ( u_\Omega , y_\Omega ) \|_{L^p (\Omega)} \: .
    \label{eq:dual_norm_as_best_approximation_symmetric_constraint}
  \end{equation}
  \label{thrm-dual_norm_as_best_approximation_symmetric_constraint}
\end{theorem}

\begin{proof}
  For all $\chi \in \Theta_{\perp,{\rm RM},\parallel}^{p^\ast}$ and for all $T \in \Sigma^{p,0}_{f,j,{\rm sym}}$, $\beta \in \R$,
  (\ref{eq:relation_shape_derivative_tensor}) together with (\ref{eq:annihilator}) implies
  \begin{equation}
    \begin{split}
      \frac{J^\prime (\Omega) [\chi]}{\| \varepsilon (\chi) \|_{L^{p^\ast} (\Omega)}}
      & = \frac{( K ( u_\Omega , y_\Omega ) - T , \varepsilon (\chi) )}{\| \varepsilon (\chi) \|_{L^{p^\ast} (\Omega)}} \\
      = & \frac{( K ( u_\Omega , y_\Omega ) - T , \varepsilon (\chi) ) - \beta ( Z^\perp , \varepsilon (\chi) )}{\| \varepsilon (\chi) \|_{L^{p^\ast} (\Omega)}} \\
      = & \frac{( K ( u_\Omega , y_\Omega ) - T - \beta Z^\perp , \varepsilon (\chi) )}{\| \varepsilon (\chi) \|_{L^{p^\ast} (\Omega)}}
      \leq \| K ( u_\Omega , y_\Omega ) - T - \beta Z^\perp \|_{L^p (\Omega)}
    \end{split}
    \label{eq:dual_inequality_tangential}
  \end{equation}
  by H\"older's inequality, which proves that the left hand side in (\ref{eq:dual_norm_as_best_approximation_symmetric_constraint}) does not exceed the right hand side.

  With the same argument as in Theorem \ref{theorem-unique_S}, a best approximation $S + \alpha Z^\perp \in \Sigma^{p,0}_{f,j,{\rm sym}} + {\rm span} \{ Z^\perp \}$ is
  obtained such that $\| S + \alpha Z^\perp - K ( u_\Omega , y_\Omega ) \|_{L^p (\Omega)}$ equals the right-hand side in
  (\ref{eq:dual_norm_as_best_approximation_symmetric_constraint}). The optimality conditions for this best approximation, in analogy to (\ref{eq:KKT}), are given by
  \begin{equation}
    \begin{split}
      ( | S + \alpha Z^\perp - K (u_\Omega,y_\Omega) |^{p-2} (S + \alpha Z^\perp - K (u_\Omega,y_\Omega)) , T ) \hspace{2cm} \\
      + ( \div \: T , \theta ) + ( \as \: T , \omega ) & = 0 \: , \\
      ( | S + \alpha Z^\perp - K (u_\Omega,y_\Omega) |^{p-2} (S + \alpha Z^\perp - K (u_\Omega,y_\Omega)) , Z^\perp ) & = 0 \: , \\      
      ( \div \: S , \chi ) - ( f \: \nabla y_\Omega , \chi ) - ( j (u_\Omega) , \div \: \chi ) & = 0 \: , \\
      ( \as \: S , \gamma ) & = 0
    \end{split}
    \label{eq:KKT_constraint_pre}
  \end{equation}
  for all $T \in \Sigma^{p,0}$, $\chi \in \Theta^{p^\ast}_{\perp,{\rm RM}}$ and $\gamma \in \Xi^{p^\ast}$ with Lagrange multipliers
  $\theta \in \Theta^{p^\ast}_{\perp,{\rm RM}}$ and $\omega \in \Xi^{p^\ast}$. Since the approximation space $\Sigma^{p,0}_{f,j,{\rm sym}} + {\rm span} \{ Z^\perp \}$
  contains $\Sigma^{p,0}_{f,j,{\rm sym}}$, the one associated with the situation without constraints, the solvability of (\ref{eq:KKT_constraint_pre}) is
  still guaranteed.
  
  With $\theta \in \Theta_{\perp,{\rm RM},\parallel}^{p^\ast}$ being the Lagrange multiplier solving (\ref{eq:KKT_constraint_pre}), we obtain from
  (\ref{eq:relation_shape_derivative_tensor}) and (\ref{eq:annihilator}) that
  \begin{equation}
    J^\prime (\Omega) [\theta] = ( K ( u_\Omega,y_\Omega ) - S , \varepsilon (\theta) ) = ( K ( u_\Omega,y_\Omega ) - S - \alpha Z^\perp , \varepsilon (\theta) )
    \label{eq:numerator_pre}
  \end{equation}
  holds. Moreover, integration by parts of the first equation in (\ref{eq:KKT_constraint_pre}) together with the symmetry of $S$,
  $Z^\perp$ and $K (u_\Omega,y_\Omega)$ implies
  \begin{equation}
    \varepsilon (\theta) = | S + \alpha Z^\perp - K (u_\Omega,y_\Omega) |^{p-2} (S + \alpha Z^\perp - K (u_\Omega,y_\Omega)) \: .
    \label{eq:deformation_relation_symmetric}
  \end{equation}
  Inserting (\ref{eq:deformation_relation_symmetric}) into (\ref{eq:numerator_pre}) leads to
  \begin{equation}
    \begin{split}
      J^\prime (\Omega) [\theta] & = - ( S + \alpha Z^\perp - K (u_\Omega,y_\Omega) , \\
      & \hspace{2.5cm} | S + \alpha Z^\perp - K (u_\Omega,y_\Omega) |^{p-2} \: (S + \alpha Z^\perp - K (u_\Omega,y_\Omega)) ) \\
      & = - \| S + \alpha Z^\perp - K (u_\Omega,y_\Omega) \|_{L^p (\Omega)}^p \: ,
    \end{split}
    \label{eq:numerator_constraint}
  \end{equation}
  while taking norms on both sides of (\ref{eq:deformation_relation_symmetric}) gives
  \begin{equation}
    \begin{split}
      & \| \varepsilon (\theta) \|_{L^{p^\ast} (\Omega)}
      = \| | S + \alpha Z^\perp - K (u_\Omega,y_\Omega) |^{p-2} (S + \alpha Z^\perp - K (u_\Omega,y_\Omega)) \|_{L^{p^\ast} (\Omega)} \\
      & = \left( \int_\Omega | S + \alpha Z^\perp - K (u_\Omega,y_\Omega) |^{(p-2) p^\ast + p^\ast} dx \right)^{1/p^\ast} \\
      & = \left( \int_\Omega | S + \alpha Z^\perp - K (u_\Omega,y_\Omega) |^p dx \right)^{(p-1)/p}
      \! = \| S + \alpha Z^\perp - K (u_\Omega,y_\Omega) \|_{L^p (\Omega)}^{p-1} \: .
    \end{split}
    \label{eq:denominator_constraint}
  \end{equation}
  From (\ref{eq:numerator_constraint}) and (\ref{eq:denominator_constraint}) we deduce that
  \begin{equation}
    \frac{J^\prime (\Omega) [\theta]}{\| \varepsilon (\theta) \|_{L^{p^\ast} (\Omega)}} = - \| S + \alpha Z^\perp - K (u_\Omega,y_\Omega) \|_{L^p (\Omega)} \: ,
  \end{equation}
  which finishes the proof of (\ref{eq:dual_norm_as_best_approximation_symmetric_constraint}).
\end{proof}

\begin{remark}
  From (\ref{eq:deformation_relation_symmetric}), the second equation in (\ref{eq:KKT_constraint_pre}) and (\ref{eq:constraint_derivative_representation}), we get
  \begin{equation}
    0 = ( \varepsilon (\theta) , Z^\perp ) = C^\prime (\Omega) [\theta]
  \end{equation}
  and therefore $\theta \in \Theta^{p^\ast}_{\perp,{\rm RM},\parallel}$. With this we obtain our final version of the KKT
  conditions in the constraint case as finding $S \in \Sigma^{p,0}$, $\alpha \in \R$, $\theta \in \Theta^{p^\ast}_{\perp,{\rm RM},\parallel}$ and
  $\omega \in \Xi^{p^\ast}$ such that
  \begin{equation}
    \begin{split}
      ( | S + \alpha Z^\perp - K (u_\Omega,y_\Omega) |^{p-2} (S + \alpha Z^\perp - K (u_\Omega,y_\Omega)) , T ) \hspace{2cm} \\
      + ( \div \: T , \theta ) + ( \as \: T , \omega ) & = 0 \: , \\
      ( | S + \alpha Z^\perp - K (u_\Omega,y_\Omega) |^{p-2} (S + \alpha Z^\perp - K (u_\Omega,y_\Omega)) , Z^\perp ) & = 0 \: , \\      
      ( \div \: S , \chi ) - ( f \: \nabla y_\Omega , \chi ) - ( j (u_\Omega) , \div \: \chi ) & = 0 \: , \\
      ( \as \: S , \gamma ) & = 0
    \end{split}
    \label{eq:KKT_constraint}
  \end{equation}
  holds for all $T \in \Sigma^{p,0}$, $\chi \in \Theta^{p^\ast}_{\perp,{\rm RM},\parallel}$ and $\gamma \in \Xi^{p^\ast}$. In fact, $\alpha Z^\perp$ is the closest-point
  projection of $S - K (u_\Omega,y_\Omega)$ with respect to $L^p (\Omega;\R^{d \times d})$ onto $\Upsilon_{{\rm RM},\parallel}^{p,\perp}$.
\end{remark}

\noindent
{\bf Example 3.}
Very common in the context of shape optimization problems are volume constraints, i.e., $C (\Omega) = | \Omega | - c$ with some constant $c > 0$. In this case,
we have
\begin{equation}
  C^\prime (\Omega) [ \theta ] = ( 1 , \div \: \theta ) = ( I , \varepsilon (\theta) ) \: ,
\end{equation}
which leads to the choice $Z^\perp \equiv I \in \R^{d \times d}$. Let us consider the situation in Example 1 with an additional volume constraint. The shape tensor
$K (u_R,y_R)$ for the disk $D_R$ with radius $R$ is still given by (\ref{eq:tensor_example_1}). We can satisfy (\ref{eq:KKT_constraint}) with $\theta = 0$ and
$\omega = 0$ by choosing
\begin{equation}
  S = K (u_R,y_R) + \gamma \: I \mbox{ with } \gamma = \left\{
  \begin{array}{ll} \frac{R^2}{16} \: , \: & R \leq 1 \: , \\ \frac{1}{8} - \frac{R^2}{16} \: , \: & R \geq 1 \: , \end{array}
  \right.
  \label{eq:example_1_solution_constraint}
\end{equation}
and $\alpha = - \gamma$. Note that (\ref{eq:KKT_constraint}) is indeed satisfied for any $\gamma \in \R$ but only the above choice in
(\ref{eq:example_1_solution_constraint}) ensures that $S \cdot n = 0$ on $\partial D_R$ so that $S \in \Sigma^{p,0}_{f,j,{\rm sym}}$ holds.
In any case, we can deduce from Theorem \ref{thrm-dual_norm_as_best_approximation_symmetric_constraint} that
$\| J^\prime (\Omega_R) \|_{p^\ast;\perp,{\rm RM},\parallel}^\ast = 0$ and therefore that $D_R$ is a stationary (in fact, an optimal) shape if
its area $| D_R | = \pi R^2$ equals $c$.

\noindent
{\bf Example 4.}
Another example, often encountered in applications, are constraints on the surface area or perimeter in three or two dimensions, respectively:
$C (\Omega) = | \partial \Omega | - c$, $c > 0$. In this case, the shape derivative is given by
\begin{equation}
  C^\prime (\Omega) [\theta] = \langle 1 , \tr \: ( \nabla \theta ) - n \cdot (\nabla \theta \cdot n) \rangle
  = \langle 1 , \tr \: ( \varepsilon (\theta) ) - n \cdot (\varepsilon (\theta) \cdot n) \rangle
  = \langle I - n \otimes n , \varepsilon (\theta) \rangle
  \label{eq:surface_constraint_shape_derivative}
\end{equation}
with the duality pairing $\langle \cdot , \cdot \rangle$ on the boundary $\partial \Omega$. Since it only involves derivatives of $\theta$ in tangential directions,
we have $\tr \: ( \varepsilon (\theta) ) - n \cdot (\varepsilon (\theta) \cdot n) \in W^{-1/p^\ast,p^\ast} (\partial \Omega)$ if $\theta \in W^{1,p^\ast}$ for sufficiently
regular domains and the right-hand side in (\ref{eq:surface_constraint_shape_derivative}) is well-defined. In the two-dimensional case, the representation
(\ref{eq:constraint_derivative_representation}) of the shape derivative (\ref{eq:surface_constraint_shape_derivative}) can be derived as follows: Let $t$ be
the tangential direction (orthogonal to $n$ for each point on the boundary curve $\partial \Omega$), then (\ref{eq:surface_constraint_shape_derivative}) turns
into
\begin{equation}
    C^\prime (\Omega) [\theta] = \langle t \otimes t , \varepsilon (\theta) \rangle = \langle t \otimes t , \nabla \theta \rangle = \langle t , \nabla \theta \cdot t \rangle \: .
  \label{eq:surface_constraint_shape_derivative_rewritten}
\end{equation}
Assuming $t \in W^{1-1/p,p} (\partial \Omega;\R^2)$, we can find an extension $r \in W^{1,p} (\Omega;\R^2)$ with $\div \: r = 0$ such that
$\left. r \right|_{\partial \Omega} = t$.
Using Stokes' theorem, (\ref{eq:surface_constraint_shape_derivative_rewritten}) turns into
\begin{equation}
  \begin{split}
    C^\prime (\Omega) [\theta] = \langle r , \nabla \theta \cdot t \rangle & = ( r , \curl \: \nabla \theta ) - ( \nabla^\perp r , \nabla \theta ) \\
    & = - ( \nabla^\perp r , \nabla \theta ) = - ( \nabla^\perp r , \varepsilon (\theta) )
  \end{split}
  \label{eq:surface_constraint_shape_derivative_volume_form}
\end{equation}
(with the notation $\nabla^\perp r = ( \partial_2 r , - \partial_1 r )$),
where the last equality follows from the symmetry of $\nabla^\perp r$ due to $\div \: r = 0$. In other words, $Z^\perp = - \nabla^\perp r$ in the representation
(\ref{eq:constraint_derivative_representation}) used in Theorem \ref{thrm-dual_norm_as_best_approximation_symmetric_constraint}.
  
Let us once more consider the disk $D_R$, now with radius $R$ such that $2 \pi R = c$. The tangential field on $\partial D_R$ is given by
$t = ( - x_2 , x_1 ) / R$ which can be extended trivially to $r = ( - x_2 , x_1 ) / R$ on $D_R$. We obtain $Z^\perp = - \nabla^\perp r = - I / R$ and the same
arguments as in Example 7.3 imply that $D_R$ is also stationary among all shapes of perimeter $| \partial \Omega | = 2 \pi R$.
  
In general, one would not want to rewrite a perimeter or surface area constraint in the form (\ref{eq:surface_constraint_shape_derivative_volume_form}) as
volume integral but rather work directly with (\ref{eq:surface_constraint_shape_derivative}). This requires modifications in the shape tensor space used to
approximate $K (u_\Omega,y_\Omega)$ in Section \ref{sec-shape_tensor_best_approximation}. We plan to explore this, also in the context of introducing
Neumann boundary conditions into (\ref{eq:bvp}), in a separate contribution.

\bibliography{../../../biblio/msc_articles,../../../biblio/msc_books}
\bibliographystyle{AIMS}

\end{document}